\documentclass[5pt]{article}
\usepackage{amsmath}
\usepackage[latin1]{inputenc}
\usepackage{amsmath,amsthm,amssymb}
\usepackage{graphicx}
\usepackage[all,poly,knot]{xy}

\usepackage{color}

\newtheorem{theorem}{Theorem}[section]
\newtheorem{lemma}{Lemma}[section]

\newtheorem{conjecture}{Conjecture}[section]

\newtheorem{remark}{Remark}[section]
\newtheorem{open problem}{Open problem}[section]
\newtheorem{proposition}{Proposition}[section]

\begin{document}

\title{ On The St$\ddot{o}$rmer Theorem And Its Applications--\\
{\small An extended version of the paper ``On some equations related to Ma's conjecture" in Integers 11(2011)}}
\author{Pingzhi Yuan \\
 School of Mathematics, Huanan Normal University\\
  Guangzhou 510631, P. R. China\\  email: mcsypz@zsu.edu.cn\\
  Jiagui Luo \\
College of Mathematics and Information Science\\
Zhaoqing University \\
Zhaoqing 526061, P. R. China\\ email: Luojg62@yahoo.com.cn\\
Alain Togb\'e\\
Mathematics Department, Purdue University North Central\\
1401 S. U.S. 421 Westville IN 46391, USA\\
email: atogbe@pnc.edu}

\date{\today}
\maketitle

\edef \tmp {\the \catcode`@}
  \catcode`@=11
  \def \@thefnmark {}
    \@footnotetext {The first and the second authors are supported by the Guangdong
    Provincial Natural Science Foundation (No. 8151027501000114) and  NSF of China (No. 12171163).
    The third author is partially supported by Purdue University North Central}

\begin{abstract} In the present paper we introduce old and new results related to
St$\ddot{o}$rmer theorem about Pell equations. Moreover we give four types of applications of these results.
\end{abstract}

% \catcode`@=\tmp
% \let\tmp = \undefined

\vskip 3mm

%--------------------------------------------------------------------%
\section{\bf Introduction}\label{sec:1}

There are many papers studying positive integer solutions,
minimal positive solutions of the diophantine equations
\begin{equation}\label{eq:1}
 kx^2-ly^2=C,\, C=1,\,2,\,4
\end{equation}
and relations between these solutions. Throughout this paper,
we assume that $k,l$ are coprime positive integers and $kl$
nonsquare; and let $2\nmid kl$ when $C=2$ or $4$. It is well known
that an interesting result is called St$\ddot{o}$rmer theorem
\cite{Di20}. More new results extending St$\ddot{o}$rmer theory had
been obtained over the years. We will attempt to list previous
known results which are related to St$\ddot{o}$rmer theorem, questions, and authors that obtained these results.
 Although in view of the duplication mentioned earlier, inadvertent omission is quite possible. Furthermore,
 we will also give and prove some new results. In Section~\ref{sec:2}, we introduce some notions and lemmas
 which will be used in the rest of the paper. In Section~\ref{sec:3}, old and new results related to
 St$\ddot{o}$rmer theorem will be investigated. Section~\ref{sec:4} is devoted to a discussion on applications
 of St$\ddot{o}$rmer theorem. First, we will discuss the solvability of Diophantine equation~\eqref{eq:1},
  where $kl=D$ is a given positive integer. Furthermore, a theorem of Perron \cite{perr} will be generalized.
  See Subsection \ref{subsec:4.1}. Second Sierpinski's conjecture will be discussed in Subsection \ref{subsec:4.2}.
   We recall the results on the generalization of Sierpinski's conjecture obtained by the first two authors.
   Then in the third subsection, Ma's conjecture will be introduced. Moreover new results related to Ma's conjecture will be proved. Finally, in Subsection \ref{subsec:4.4}, using St$\ddot{o}$rmer theorem and its extensions, we will give all integer solutions of other Diophantine equations particularly related to a result of Ljunggren.

%--------------------------------------------------------------------%

\section{Notions and lemmas}\label{sec:2}

We recall that the minimal positive solution of Diophantine equation \eqref{eq:1} is one
of all positive integer solutions $(x,y)$ of equation~\eqref{eq:1} such that $x\sqrt{k}+y\sqrt{l}$ is the smallest.
 One can easily see that this is equivalent to determine a positive integer solution $(x,y)$ of equation~\eqref{eq:1}
 such that $x$ and $y$ are the smallest. If $k=C=1$ or $l=C=1$, then such a solution is also called the fundamental
 solution of equation \eqref{eq:1}.

 Throughout this paper, we denote by $\mathbb{Z},\, \mathbb{N}$ the set of integers and the set of positive integers respectively.
 If $k=C=1$ or $k=1,C=4$ , $x_1+y_1\sqrt{l}$ is the fundamental solution of equation \eqref{eq:1}, then
 we have the following result.
\begin{lemma}{\rm (\cite{suny})}\label{lem:2.1}  All positive integer solutions of equation~\eqref{eq:1} are given by
$$\frac{x+y\sqrt{l}}{\sqrt{C}} = \left(\frac{x_1+y_1\sqrt{l}}{\sqrt{C}}\right)^n,\; n\in\mathbb{N}.$$
 \end{lemma}

If $k>1$ or $C=2$, $x_1\sqrt{k}+y_1\sqrt{l}$ is the minimal positive solution of equation~\eqref{eq:1}, then we have the following lemma.
 \begin{lemma}{\rm (\cite{luo}, \cite{suny})}\label{lem:2.2}  All positive integer solutions of equation~\eqref{eq:1} are given by
$$\frac{x\sqrt{k}+y\sqrt{l}}{\sqrt{C}} = \left[\frac{x_1\sqrt{k}+y_1\sqrt{l}}{\sqrt{C}}\right]^n,\; n\in\mathbb{N}, \; 2\nmid n.$$
 \end{lemma}

Let $R>0$,  $Q$ be nonzero coprime integers with $R-4Q>0$. Let
$\alpha$ and $\beta$ be the two roots of the trinomial
$x^2-\sqrt{R}x+Q$.
 The Lehmer sequence $\{P_n(R, Q)\}$ and the associated Lehmer sequence $\{Q_n(R, Q)\}$
 with parameters $R$ and $Q$ are defined as follows:\\
$$P_n=P_n(R,Q)=\left\{\begin{array}{cc}(\alpha^n-\beta^n)/(\alpha-\beta),
\quad 2\nmid n,
 \\(\alpha^n-\beta^n)/(\alpha^2-\beta^2),2|n \end{array} \right.$$
and
$$Q_n=Q_n(R,Q)=\left\{\begin{array}{cc}(\alpha^n+\beta^n)/(\alpha+\beta),
\quad 2\nmid n,
 \\\alpha^n+\beta^n,\quad \quad \quad \quad 2|n \end{array}
 \right.$$

For simplicity, in this paper we denote $(\alpha^{dr}-\beta^{dr})/(\alpha^{d}-\beta^{d})$ and
 $(\alpha^{r}-\beta^{r})/(\alpha-\beta)$ by $P_{r,d}$ and $P_r$ respectively. \\

Lehmer sequences and associated Lehmer sequences have many
interesting properties and often raise in the study of exponential
Diophantine equations. It is not difficult to see that $P_n$ and
$Q_n$ are both positive integers for all positive integers $n$. The
details can be seen in \cite{lehm}, \cite{ribe}, \cite{yuan1}, \cite{Yuan:2004-ProcAmer}.

\begin{proposition}\label{prop:Lehmer}
Let $d=\gcd(m,n)$ for some integers $m$ and $n$. We have
\begin{enumerate}
\item \label{item:Lehmer-1} If $P_m\neq 1$, then $P_m|P_n$ if and only if $m|n$.

\item \label{item:Lehmer-2} If $m \geq 1$, then $Q_m|Q_n$ if and only if $n/m$ is an odd integer.

\item \label{item:Lehmer-3} $\gcd(P_m,P_n)=U_d$.

\item \label{item:Lehmer-4} $\gcd(Q_m, Q_n)=Q_d$ if $m/d$ and $n/d$ are odd, and 1 otherwise.

\item \label{item:Lehmer-5} $\gcd(P_m, Q_n)=Q_d$ if $m/d$ is even, and 1 otherwise.

\item \label{item:Lehmer-6} $P_{2m}=2P_m Q_m$.

\item \label{item:Lehmer-7} For any prime $p$, $ord_p(P_{mp}/P_m)=1$ or $0$ which depends on $p|P_ma$ or not.

%\item \label{item:Lehmer-6} $\gcd(P_m, Q_m)=1$.
\end{enumerate}
\end{proposition}

\begin{lemma}{\rm (\cite{luoy})}\label{lem:2.3} Assume that $R$ and $Q$ are all odd. If $Q_n=ku^2,\, k|n,$ then $n=1, 3, 5.$ If $Q_n=2ku^2,\, k|n,$ then $n=3.$
 \end{lemma}
Let $k>1$ and assume that the Diophantine equation \eqref{eq:1} has solution in positive integers.
 Let $x_1\sqrt{k}+y_1\sqrt{l}$ be the minimal positive solution of equation \eqref{eq:1} and define
$$\alpha=\frac{x_1\sqrt{k}+y_1\sqrt{l}}{\sqrt{C}},\quad \beta=\frac{x_1\sqrt{k}-y_1\sqrt{l}}{\sqrt{C}}.$$
Then $R=(\alpha+\beta)^2,\, 4kx_1^2|C,\, Q=\alpha\beta=1$ and $R-4Q>0$. Furthermore, for $n$ odd, we define
$$\alpha^n=\frac{x_n\sqrt{k}+y_n\sqrt{l}}{\sqrt{C}} = \left(\frac{x_1\sqrt{k}+y_1\sqrt{l}}{\sqrt{C}}\right)^n.$$
By Lemma~\ref{lem:2.2}, we know that all positive integer solutions $(X,Y)$ of equation \eqref{eq:1}
are of the form $(x_n,y_n)$. However, for $n$ even, we cannot define $\alpha^n$ as above.
In fact, if $(x,y)$ is a solution of Diophantine equation
\begin{equation}\label{eq:2}
 kx^2-ly^2=1,\,k>1,
\end{equation}
then it is not difficult to see that
$(x\sqrt{k}+y\sqrt{l})^2=2kx^2-1+2xy\sqrt{kl}.$ Hence $(2kx^2-1,2xy)$ is a solution of the Diophantine equation
\begin{equation}\label{eq:3}
 x^2-kly^2=1.
\end{equation}
Similarly, if $x\sqrt{k}+y\sqrt{l}$ and  $u\sqrt{k}+v\sqrt{l}$ are
solutions of the Diophantine equations
\begin{equation}\label{eq:4}
 kx^2-ly^2=2,
\end{equation}
 and
 \begin{equation}\label{eq:5}
 kx^2-ly^2=4,\,k>1,
 \end{equation}
 respectively. Then
 $\left(\frac{x\sqrt{k}+y\sqrt{l}}{\sqrt{2}}\right)^2 = kx^2-1+xy\sqrt{kl}$ and
 $\left(\frac{u\sqrt{k}+v\sqrt{l}}{2}\right)^2 = \frac{lv^2+1}{2}u\sqrt{k}+\frac{kx^2-1}{2}v\sqrt{l}$
 are solutions of equations \eqref{eq:3} and \eqref{eq:2} respectively. But
 $2\left(\frac{u\sqrt{k}+v\sqrt{l}}{2}\right)^2 = ku^2-2+uv\sqrt{kl}$ is a solution of the Diophantine equation
 \begin{equation}\label{eq:6}
 x^2-kly^2=4.
 \end{equation}
A solution $X+Y\sqrt{kl}$ of equation \eqref{eq:6} implies that
$\left(\frac{X+Y\sqrt{kl}}{2}\right)^3$ is a solution of equation \eqref{eq:3}. Note that if $(x,y)$ is a solution of equation \eqref{eq:2}, then $(2x,2y)$ is a solution of equation \eqref{eq:5}. So we assume that equation \eqref{eq:5} has solutions in odd positive integers.

One can ask whether or not there exist relations between the minimal positive solutions of above equations. The answer is positive
and the relations are given by the following result.
 \begin{lemma}{\rm (\cite{yluo})}\label{lem:2.4} Let
 $\varepsilon_1, \varepsilon_2, \varepsilon_3, \varepsilon_4, \varepsilon_5$ be the minimal positive solutions of equation \eqref{eq:2}-\eqref{eq:6} respectively. Then
 \begin{enumerate}
 \item $\varepsilon_2=\varepsilon_1^2$;

 \item $\varepsilon_2=\varepsilon_3^2/2$;

 \item $\varepsilon_2=(\varepsilon_4/2)^6$;

 \item $\varepsilon_2=(\varepsilon_5/2)^3$;

 \item $\varepsilon_1=(\varepsilon_4/2)^3$;

 \item $\varepsilon_5=\varepsilon_4^2/2$.
 \end{enumerate}

\end{lemma}

%--------------------------------------------------------------%

\section{St$\ddot{o}$rmer theorem and its extensions}\label{sec:3}

St$\ddot{o}$rmer obtained an important property on Pell
equations, called St$\ddot{o}$rmer theorem and stated as follows.
 \begin{theorem}{\rm (St$\ddot{o}$rmer theorem \cite{Di20})} \label{Thm:3.1} Let $D$ be a positive nonsquare integer. Let $(x_1, \, y_1)$ be a positive integer solution of Pell equation
 \begin{equation}\label{eq:7}
 x^2-Dy^2=\pm 1.
 \end{equation}
 If every prime divisor of $y_1$ divides $D$, then $x_1+y_1\sqrt{D}$ is the fundamental solution.
 \end{theorem}
Considering Diophantine equation \eqref{eq:2}, D.T. Walker
\cite{wal67} obtained a result similar to St$\ddot{o}$rmer theorem.
See also Q. Sun and P. Yuan \cite{suny}
  \begin{theorem}{\rm (\cite{wal67,suny})} \label{Thm:3.2} Let $(x,y)$
  be a positive integer solution of Diophantine equation \eqref{eq:2}.\\
(i) If every prime divisor of $x$  divides $k$, then
$$x\sqrt{k}+y\sqrt{l}=\varepsilon$$
or
$$x\sqrt{k}+y\sqrt{l}=\varepsilon^3,$$
and $x=3^sx_1,\,\,3\nmid x_1,\,\,3^s+3=4kx_1^2,$ where
$\varepsilon=x_1\sqrt{k}+y_1\sqrt{l}$ is the minimal positive
solution of equation \eqref{eq:2}, $s\in \mathbb{N}$.\\
 (ii) If every prime divisor of $y$ divides $l$, then
$$x\sqrt{k}+y\sqrt{l}=\varepsilon$$
or
$$x\sqrt{k}+y\sqrt{l}=\varepsilon^3,$$
and $y=3^sy_1,\,\,3\nmid y_1,\,\,3^s-3=4ly_1^2,$  $s\in \mathbb{N}$.
 \end{theorem}

On Diophantine equations \eqref{eq:4} and \eqref{eq:5}, using the
method in \cite{suny}, the second author proved the following
results.
  \begin{theorem}{\rm (\cite{luo})} \label{Thm:3.3} Let $(x,y)$
  be a positive integer solution of Diophantine equation \eqref{eq:4}.\\
(i) If every prime divisor of $x$  divides $k$, then
$$x\sqrt{k}+y\sqrt{l}=\varepsilon$$
or
$$\frac{x\sqrt{k}+y\sqrt{l}}{\sqrt{2}}= \left(\frac{\varepsilon}{\sqrt{2}}\right)^3,$$
and $x=3^sx_1,\,\,3\nmid x_1,\,\,3^s+3=2kx_1^2,$ where
$\varepsilon=x_1\sqrt{k}+y_1\sqrt{l}$ is the minimal positive
solution of equation \eqref{eq:4}, $s\in \mathbb{N}$.\\
 (ii) If every prime divisor of $y$ divides $l$, then
$$x\sqrt{k}+y\sqrt{l}=\varepsilon$$
or
$$\frac{x\sqrt{k}+y\sqrt{l}}{\sqrt{2}}= \left(\frac{\varepsilon}{\sqrt{2}}\right)^3,$$
and $y=3^sy_1,\;3\nmid y_1,\;3^s-3=2ly_1^2,$  $s\in \mathbb{N}$.
 \end{theorem}

 \begin{theorem}{\rm (\cite{luo})} \label{Thm:3.4} Let $(x,y)$ be a positive integer solution of Diophantine equation \eqref{eq:5}.\\
(i) If every prime divisor of $x$  divides $k$ , then
$\varepsilon=x\sqrt{k}+y\sqrt{l}$ is the minimal positive solution
of equation \eqref{eq:5} except for the case  $(k,l,x,y)=(5,1,5,11).$\\
 (ii) If every prime divisor of $y$ divides $l$ , then $\varepsilon=x\sqrt{k}+y\sqrt{l}$ is the minimal positive solution of equation \eqref{eq:5}.
 \end{theorem}
\begin{remark}\label{rmk:3.1} From the proofs of Theorems \ref{Thm:3.2}, \ref{Thm:3.3}, \ref{Thm:3.4} in \cite{luo},
\cite{wal67},  \cite{suny}, one can easily observe that the above
Theorems are also true if every prime divisor of $x$ divides $k$ or
$x_1$ and if every prime divisor of $y$ divides $l$ or $y_1$.
\end{remark}

Hanfei Mei, Long Mei, Qiyi Fan, Wei Song showed the following theorem.
\begin{theorem}{\rm (\cite{mei01})} \label{Thm:3.5} Let $D$ be a positive nonsquare integer. Let $(x, \, y)$ be a positive integer solution of Pell equation
 \begin{equation}\label{eq:8}
 x^2-Dy^2= 1,
 \end{equation}
 with $y=p^ny'$, where $p$ is a prime not dividing $D$ and $n\in \mathbb{N}$. If every prime divisor of $y'$ divides $D$, then $x+y\sqrt{D}=\varepsilon$ or $\varepsilon^2$ or $\varepsilon^3$, where $\varepsilon=x_1+y_1\sqrt{D}$ is the fundamental solution of equation \eqref{eq:8}.
 \end{theorem}
 \begin{remark}\label{rmk:3.2} If $x+y\sqrt{D}=\varepsilon^3$, then $y=3^sp^ny_1,s\in \mathbb{N},s>1$ except for $(x,y,D)=(26,15,3)$.
 \end{remark}

Furthermore, Hanfei Mei, Shenfang Sun proved the following result.
\begin{theorem}{\rm (\cite{mei02})}\label{thm:3.6} Let $(x, \, y)$
 be a positive integer solution of Pell equation \eqref{eq:8}
 with $y=p_1^{n_1}p_2^{n_2}y'$, where $p_1,p_2$ are any primes not dividing $D$ and $n_1,n_2\in \mathbb{N}$. If every prime divisor of $y'$ divides $D$, then $x+y\sqrt{D}=\varepsilon$ or $\varepsilon^2$ or $\varepsilon^3$ or $\varepsilon^4$ or $\varepsilon^6$ or $\varepsilon^{q^r}$, where $\varepsilon=x_1+y_1\sqrt{D}$ is the fundamental solution of \eqref{eq:8}, $q$ an odd prime, and $r\in \mathbb{N}$.
 \end{theorem}
\begin{remark}\label{rmk:3.3}
We can prove that $r=1$ and $(q,p_1p_2)=1$.
\end{remark}
On Diophantine equations
\begin{equation}\label{eq:9}
 x^2-Dy^2= -1,
 \end{equation}
 \begin{equation}\label{eq:10}
 x^2-Dy^2= 4,
 \end{equation}
 \begin{equation}\label{eq:11}
 kx^2-ly^2= 4,
 \end{equation}
similar results can be obtained. In fact, we have a series of results.
\begin{theorem}\label{thm:3.7} Let $D$ be a positive nonsquare integer. Let $(x,\, y)$ be a positive integer solution of Pell equation \eqref{eq:9} with $y=p^ny'$, where $p$ is a prime not dividing $D$ and $n\in \mathbb{N}$. If every prime divisor of $y'$ divides $D$, then $x+y\sqrt{D}=\varepsilon$ or $\varepsilon^q$, where $\varepsilon=x_1+y_1\sqrt{D}$ is the fundamental solution of equation \eqref{eq:9} and $q$ an odd prime with $(p,q)=1$.
\end{theorem}

\begin{proof} It is easy to see that the result is true if $p|y_1$
by Remark \ref{rmk:3.1}. Now we assume that $p\nmid y_1$. By
Lemma \ref{lem:2.2}, we know that
\begin{equation}\label{eq:12}
 x+y\sqrt{D}=\left(x_1+y_1\sqrt{D}\right)^m,
 \end{equation}
 for some odd integer $m$. If $m=1$, there is nothing to do. Hence
 we assume $m>1$. We write $m=m_1q^r$, where $q$ is a prime divisor of $m$ and $(m_1,q)=1, r\in \mathbb{N}.$ We claim that $p|P_q$. Otherwise every prime divisor of $y_q=y_1P_q$ divides $D$ by the assumption. It follows that $q=1$ by Theorem \ref{Thm:3.1}. This is a contradiction. By Proposition~\ref{prop:Lehmer} we know that
 $(P_{m_1},P_q)=P_{(m_1,q)}=P_1=1.$ It implies that every prime
 divisor of $y_{m_1}=y_1P_{m_1}$ divides $D$ since  $y_{m_1}|y_m=y=p^ny'.$ Thus $m_1=1$ again by Theorem 3.1, and so
 $m=q^r$. It is obvious that $q\neq p$ since
 $$P_q=\sum_{r=0}^{(q-1)/2}\binom
 q{2r+1}x_1^{q-2r-1}(Dy_1^2)^r.$$
If $r>1$, then
 \begin{equation}\label{eq:13}
 P_{q,q}=\sum_{r=0}^{(q-1)/2}\binom
 q{2r+1}x_q^{q-2r-1}(Dy_q^2)^r.
\end{equation}
Since $(P_q,P_{q,q})|q$, hence $p\nmid P_{q,q}$ and thus every prime divisor $P$ of $P_{q,q}$ divides $D$. Then from equation \eqref{eq:13} we get $P|qx_q^{q-1}$, and so $P=q$. If $q>3$, we claim that $P_{q,q}=q$. Otherwise, from equation \eqref{eq:13} we have $q^2|qx_q^{q-1}$ which is impossible. On the other hand, if $q>3$, then
$$P_{q,q}=\sum_{r=0}^{(q-1)/2}\binom
 q{2r+1}x_q^{q-2r-1}(Dy_q^2)^r>q.$$
 This is a contradiction. So we obtain $q=3$, whence $3|D$. Since $x_q^2-Dy_q^2=-1$, we get $-1=(-1|3)=1$. This is a contradiction. Thus $r=1$. This completes the proof of Theorem \ref{thm:3.7}.
 \end{proof}

\begin{theorem}\label{thm:3.8} Let $(x, \, y)$ be a positive integer solution of Pell equation \eqref{eq:9} with $y=p_1^{n_1}p_2^{n_2}y'$, where $p_1, p_2$ are primes not dividing $D$ and $n_1, n_2\in \mathbb{N}$.  If every prime divisor of $y'$ divides $D$, then $x+y\sqrt{D}=\varepsilon$ or   $\varepsilon^q$ or $\varepsilon^{q^2}$, where $\varepsilon=x_1+y_1\sqrt{D}$ is the fundamental solution of equation \eqref{eq:9}, $q$ an odd prime with $(q,p_1p_2)=1$.
 \end{theorem}

\begin{proof} By Lemma \ref{lem:2.2}, we know that
\begin{equation}\label{eq:14}
 x+y\sqrt{D}=\left(x_1+y_1\sqrt{D}\right)^m
 \end{equation}
 for some odd integer $m$. It is enough to prove the result for $m>1$ and $p_1\nmid y_1,\, p_2\nmid y_1.$ Since $y_1|y=y_m = p_1^{n_1}p_2^{n_2}y'$, so we have that $y'=y_1u.$ Write $m=m_1q^r$, where $q$ is a prime divisor of $m$ and $(m_1, q)=1, r\in \mathbb{N}.$ If $m_1>1$, using what we proved in Theorem \ref{thm:3.7}, we know that $(m, p_1p_2)=1$  and $p_1|P_{m_1},\, p_2|P_{q^r}$ or $p_2|P_{m_1},\, p_1|P_{q^r}$. Without loss of generality, we assume that $p_1|P_{m_1},\, p_2|P_{q^r}$.
 Then since $(P_{m_1}, P_{q^r})=P_{(m_1,q^r)}=P_1=1$, $m_1$ is an odd prime. Say $m_1=p$ and $r=1$ by Theorem \ref{thm:3.7}. Therefore $m=pq$. From
\begin{equation}\label{eq:15}
 y_1uP_1^{n_1}P_2^{n_2}=y=y_{pq}=y_qP_{p,q},
 \end{equation}
 we get
 \begin{equation}\label{eq:16}
 uP_1^{n_1}P_2^{n_2}=P_qP_{p,q}.
 \end{equation}
Since $(P_q,P_{p,q})|p$ and $p\nmid P_q$, so $(P_q,P_{p,q})=1.$
Thus
\begin{equation}\label{eq:17}
 P_{p,q}=u_1P_1^{n_1},\; P_q=u_2P_2^{n_2},\; u=u_1u_2.
 \end{equation}
 Note that
 \begin{equation}\label{eq:18}
 P_q=\sum_{r=0}^{(q-1)/2}\binom
 q{2r+1}x_1^{q-2r-1}(Dy_1^2)^r
\end{equation}
and
\begin{equation}\label{eq:19}
 P_{p,q}=\sum_{r=0}^{(p-1)/2}\binom
 p{2r+1}x_q^{p-2r-1}(Dy_q^2)^r.
\end{equation}
Using an argument similar to that in Theorem \ref{thm:3.7}, we have $u_1=1$ or $p$ and $u_2=1$ or $q$. And from
\begin{equation}\label{eq:20}
 uP_1^{n_1}P_2^{n_2}=P_pP_{q,p},
 \end{equation}
and $p\nmid P_{q,p}$, we obtain $P_p=u_1p_1^{n_1}=P_{p,q}$, which is impossible. Therefore $m_1=1$ and $m=q^r$. If $r>3$, then we have $p_1|P_{q^2},\; p_2|P_{q^2}$ by Theorem \ref{thm:3.7} and
\begin{equation}\label{eq:21}
 P_{q,q^2}=\sum_{r=0}^{(q-1)/2}\binom
 q{2r+1}x_{q^2}^{q-2r-1}(Dy_{q^2}^2)^r.
\end{equation}
As $(P_{q,q^2},P_{q^2})|q$, hence $p_1\nmid P_{q,q^2},\; p_2\nmid P_{q,q^2}$. Therefore, every prime divisor $P$ of $P_{q,q^2}$ divides $D$. Then from equation \eqref{eq:21} we get $P|qx_{q^2}^{q-1}$ so $P=q$. If $q>3$, we affirm that $P_{q,q^2}=q$. Otherwise from equation \eqref{eq:21}, we have $q^2|qx_{q^2}^{q-1}$, which is impossible. On the other hand, if $q>3$, then
$$P_{q,q^2}=\sum_{r=0}^{(q-1)/2}\binom
 q{2r+1}x_{q^2}^{q-2r-1}(Dy_{q^2}^2)^r>q.$$
 This is a contradiction. So we obtain $q=3$, thus $3|D$. Since $x_q^2-Dy_q^2=-1$, we get $-1=(-1|3)=1$. It also leads to a contradiction. Thus $r\leq 2$. This completes the proof.
\end{proof}

 \begin{theorem}\label{thm:3.9} Let $D$ be a positive nonsquare integer such that the Diophantine equation \eqref{eq:10} is solvable in odd integers $x$ and $y$. Let $(x, \, y)$ be a positive integer solution of Pell equation \eqref{eq:10} with $y=p^ny'$, where $p$ is a prime not dividing $D$ and $n\in \mathbb{N}$. If every prime divisor of $y'$ divides $D$, then $\frac{x+y\sqrt{D}}{2}=\frac{\varepsilon}{2}$ or $(\frac{\varepsilon}{2})^2$ or  $(\frac{\varepsilon}{2})^3$ except for the case $(x,y,D)=(123,55,5)$, where $\varepsilon=x_1+y_1\sqrt{D}$ is the minimal positive solution of \eqref{eq:10}.
 \end{theorem}

\begin{proof} It is easy to see that the result is true if $p|y_1$ by Remark \ref{rmk:3.1}. We assume that $p\nmid y_1$. By Lemma \ref{lem:2.1} we know that
\begin{equation}\label{eq:22}
 \frac{x+y\sqrt{D}}{2}=\left(\frac{x_1+y_1\sqrt{D}}{2}\right)^m
 \end{equation}
 for some positive integer $m$. If $m=1$, there is nothing to do. Hence we take $m>1$.\\

{\bf Case 1:} We assume $2|m$. We write $m=2m_1$. From equation \eqref{eq:22} we get
\begin{equation}\label{eq:23}
 \frac{x+y\sqrt{D}}{2}=\left(\frac{x_{m_1}+y_{m_1}\sqrt{D}}{2}\right)^2.
 \end{equation}
 Hence $$x_{m_1}y_{m_1}=p^ny'.$$ Since $x_{m_1}^2-Dy_{m_1}^2=4,$ we
 have that $(x_{m_1},y_{m_1})=1$ or $2$. If $(x_{m_1},y_{m_1})=1$,
 then $y_{m_1}=y'$. It follows that every prime divisor of $y_{m_1}$ divides $D$. By Theorem \ref{Thm:3.4}, we obtain $m_1=1$, whence $m=2.$ If $(x_{m_1},y_{m_1})=2$, then $p=2$, we have $x_{m_1}=2^{n-1}$ and $Q_{m_1}=2^t.$ By Lemma \ref{lem:2.3} and Proposition~\ref{prop:Lehmer} we have $m_1=3.$ This implies that $x_1|2^{n-1}=x_3$, which is impossible since $x_1$ is an odd $>1$.\\

 {\bf Case 2:} Now we assume $2\nmid m$. We write $m=m_1q^r$, where $q$ is a prime divisor of $m$,  $(m_1,q)=1,\; r\in \mathbb{N}.$ We claim that $p|P_q$. Otherwise every prime divisor of $y_q=y_1P_q$ divides $D$ by the assumption. It follows that $q=1$ by Theorem \ref{Thm:3.4}. This leads to a contradiction. By Proposition~\ref{prop:Lehmer}, we know that $(P_{m_1},P_q)=P_{(m_1,q)}=P_1=1.$ This implies that every prime divisor of $y_{m_1}=y_1P_{m_1}$ divides $D$ since
 $y_{m_1}|y_m=y=p^ny'.$ Thus we obtain $m_1=1$ by Theorem \ref{Thm:3.4}. Therefore we have $m=q^r$. It is obvious that $q\neq p$ since
 \begin{equation}\label{eq:24}
 P_q=\sum_{r=0}^{(q-1)/2}\binom
 q{2r+1}(x_1/2)^{q-2r-1}(Dy_1^2/4)^r.
 \end{equation}
If $r>1$, then from
 \begin{equation}\label{eq:25}
 P_{q,q}=\sum_{r=0}^{(q-1)/2}\binom
 q{2r+1}(x_q/2)^{q-2r-1}(Dy_q^2/4)^r.
\end{equation}
We will prove that $q=3$ $P_{q,q}=3^t$, and $3|D$. Since
$P_3=(3x_1^2+Dy_1^2)/4$, we get $3|P_3$ so $3|y_3$. If $t=1$, from equation \eqref{eq:25} we have $12=3x_3^2+Dy_3^2=4Dy_3^2+12$. It follows that $4Dy_3^2=0$ which is impossible. If $t>1$, again from equation \eqref{eq:25} we see that $9|3x_3^2$, which is also impossible. We have shown that $r=1$. If $q>3$, we claim that $P_q\neq p^n$. Otherwise if $q\equiv 1\pmod{4}$. We write $q=4k+1.$ From
$$P_q=P_{4k+1}=(P_{2k+1}-P_{2k})(P_{2k+1}+P_{2k}),$$
and $((P_{2k+1}-P_{2k}),(P_{2k+1}+P_{2k}))=1$, since
$$((P_{2k+1}-P_{2k}),(P_{2k+1}+P_{2k}))|2(P_{2k},P_{2k+1})=2,\; 2\nmid P_q,$$
we have $P_{2k+1}-P_{2k}=1$, which is impossible. If $q\equiv 3\pmod{4}$, we write $q=4k+3.$ From
$$P_q=P_{4k+3}=(P_{2k+2}-P_{2k+1})(P_{2k+2}+P_{2k+1}),$$
and $((P_{2k+2}-P_{2k+1}),(P_{2k+2}+P_{2k+1}))=1$ since
$$((P_{2k+2}-P_{2k+1}),(P_{2k+2}+P_{2k+1}))|2(P_{2k+2},P_{2k+1})=2,2\nmid P_q,$$
we obtain $P_{2k+2}-P_{2k+1}=1$ which is impossible. Let $P\neq p$ be an arbitrary prime divisor of $P_q$. Then from equation \eqref{eq:24}, we can easily prove that $P=q$ and
$$q^2\nmid P_q.$$
Thus we have shown that $P_q=qp^n$ if $q>3.$  On the other hand, if $q>5$, then when $q\equiv 1\pmod{4}$, we write $q=4k+1.$ We have $k>2$ and
$P_{2k+1}-P_{2k}=q$ or $P_{2k+1}+P_{2k}=q.$ But
$P_s=x_1P_{s-1}-P_{s-2},s\geq 2.$ Hence we get
$P_s-P_{s-1}=(x_1-1)P_{s-1}-P_{s-2}>2(P_{s-1}-P_{s-2}),s\geq 3$. Thus we have $P_{2k+1}-P_{2k}>2^{2k}=4^k>4k+1=q,k\geq 3.$ This leads to a contradiction. When $q\equiv 3\pmod{4}$, we write $q=4k+3.$ We have $k\geq 1$ and
$P_{2k+2}-P_{2k+1}=q$ or $P_{2k+2}+P_{2k+1}=q.$ But
$P_{2k++2}-P_{2k=1}>2^{2k+1}>4k+3=q,k\geq 1.$ This also leads to a contradiction. Therefore we have shown that $q=3$ or $q=5$. If $q=5$, by what we proved before we have $P_3-P_2=5.$ Since
$P_0=0,P_1=1,P_2=x_1,P_3=x_1^2-1,$ so we obtain $x_1(x_1-1)=6.$ Thus $x_1=3.$ From $x_1^2-Dy_1^2=4,$ we get $Dy_1^2=5$. It follows that $D=5,y_1=1.$ Therefore
$y=y_5=P_5=(P_3-P_2)(P_3+P_2)=(8-3)(8+3)=55.$ This completes the proof.
 \end{proof}

\begin{remark}\label{rmk:3.4} If $\frac{x+y\sqrt{D}}{2}=(\frac{\varepsilon}{2})^3$,
 then $p=2$ by the Proposition~\ref{prop:Lehmer}.
 \end{remark}
\begin{theorem}\label{thm:3.10} Let $(x, \, y)$
 be a positive integer solution of Pell equation \eqref{eq:10} with $y=p_1^{n_1}p_2^{n_2}y'$, where $p_1,p_2$ are primes not dividing $D$ and $n_1,n_2\in \mathbb{N}$. If every prime divisor of
  $y'$ divides $D$, then $\frac{x+y\sqrt{D}}{2}=\frac{\varepsilon}{2}$ or $(\frac{\varepsilon}{2})^2$ or
 $(\frac{\varepsilon}{2})^3$ or $(\frac{\varepsilon}{2})^4$ or $(\frac{\varepsilon}{2})^6$ or
 $(\frac{\varepsilon}{2})^q,$ where $\varepsilon=x_1+y_1\sqrt{D}$ is the minimal positive solution of equation \eqref{eq:10}, $q$ is an odd prime with $(q, p_1p_2)=1$.
 \end{theorem}

\begin{proof} It is enough to prove the result for the case of $p_1\nmid y_1,\; p_2\nmid y_1$ according to Remark \ref{rmk:3.1}. By Lemma \ref{lem:2.1}, we know that
\begin{equation}\label{eq:26}
 \frac{x+y\sqrt{D}}{2}=\left(\frac{x_1+y_1\sqrt{D}}{2}\right)^m
 \end{equation}
 for some positive integer $m$. If $m=1$, there is nothing to do. Therefore, we take $m>1$.\\

{\bf Case 1:} We assume $2|m$. Then we write $m=2m_1$. From equation \eqref{eq:26} we get
\begin{equation}\label{eq:27}
 x_{m_1}y_{m_1}=p_1^{n_1}p_2^{n_2}y'.
 \end{equation}
 If $(x_{m_1},y_{m_1})=1$,  then $y_{m_1}=p_1^{n_1}y'$  or $y_{m_1}=p_2^{n_2}y'$. It follows that $y_{m_1}$ satisfies the condition of Theorem \ref{thm:3.9}, so we obtain $m_1=1$ or $2$ or $3$, whence $m=2$ or $4$ or $6$. If $(x_{m_1},y_{m_1})=2$, then we have $x_{m_1}=2^{n-1}$ and so $Q_{m_1}=2^t.$ By Lemma \ref{lem:2.3} and Proposition~\ref{prop:Lehmer} we obtain $m_1=3.$ This implies that $x_1|2^{n-1}=x_3$, which is impossible since $x_1$ is an odd $>1$.\\

 {\bf Case 2:} Now we assume $2\nmid m$. We put $m=m_1q^r$, where $q$ is a prime divisor of $m$ and $(m_1,q)=1,\; r\in \mathbb{N}.$ If $m_1>1$, by Theorem \ref{thm:3.9} we know that $m=15,\, D=5$. But a simple calculation shows that $y=y_{15}=2^3\cdot5\cdot11\cdot31\cdot61$.
 Thus we have now shown that $m_1=1.$
If $r>1$, then from
 \begin{equation}\label{eq:28}
 P_{q,q}=\sum_{r=0}^{(q-1)/2}\binom
 q{2r+1}(x_q/2)^{q-2r-1}(Dy_q^2/4)^r.
\end{equation}
We will prove $q=3$, $P_{q,q}=3^t,\; 3|D$. Since
$P_3=(3x_1^2+Dy_1^2)/4$, we get $3|P_3$ and $3|y_3$. If $t=1$, from equation \eqref{eq:28} we have $12=3x_3^2+Dy_3^2=4Dy_3^2+12$. It follows that $4Dy_3^2=0$ which is impossible. If $t>1$, again from \eqref{eq:28} we see that $9|3x_3^2$. This also is impossible and completes the proof of Theorem \ref{thm:3.10}.
 \end{proof}

 \begin{theorem}\label{thm:3.11} Suppose that the Diophantine equation \eqref{eq:11} is solvable
 in odd integers $x$ and $y$. Let $(x, \, y)$
 be a positive integer solution of equation \eqref{eq:11}.\\
  (i)\, If $x=p^nx'$, where $p$ is a prime not dividing $k$ and $n\in \mathbb{N}$ and every prime divisor of $x'$ divides $k$, then $\frac{x\sqrt{k}+y\sqrt{l}}{2}=\frac{\varepsilon}{2}$ or $(\frac{\varepsilon}{2})^q$ except for
  $(k,l,x,y)=(5,1,75025,167761)$, where $\varepsilon=x_1\sqrt{k}+y_1\sqrt{l}$ is the minimal positive solution of \eqref{eq:11}, $q$ an odd prime with $q\neq p.$\\
    (ii)\,If $y=p^ny'$, where $p$ is a prime not dividing $l$ and $n\in \mathbb{N}$ and every prime divisor of $y'$ divides $l$, then $\frac{x\sqrt{k}+y\sqrt{l}}{2}=\frac{\varepsilon}{2}$ or $(\frac{\varepsilon}{2})^q$, where $q$ an odd prime with $q\neq p.$
 \end{theorem}

\begin{proof} (i) It is easy to see that the result is true if $p|x_1$ by Remark \ref{rmk:3.1}. Now we assume that $p\nmid x_1$. By Lemma \ref{lem:2.2}, we know that
\begin{equation}\label{eq:29}
 \frac{x\sqrt{k}+y\sqrt{l}}{2}= \left(\frac{x_1\sqrt{k}+y_1\sqrt{l}}{2}\right)^m,
 \end{equation}
 for some odd positive integer $m$. If $m=1$, there is nothing to do. Hence we assume $m>1$. We write $m=m_1q^r$, where $q$ is a prime divisor of $m$ and $(m_1,q)=1, r\in \mathbb{N}.$ \\

 {\bf Case 1:} We assume that $p|Q_{q^r}.$  If $p\nmid Q_q$, then every prime divisor of $x_q=x_1Q_q$ divides $k$ by the assumption. It follows that $q=5,k=5,l=1$ by Theorem \ref{Thm:3.4}. But a simple calculation gives $x_{25}=75025=5^2\cdot3001$ which implies that $r=2,m_1=1.$ Therefore we have $m=25.$ If $p|Q_q$, we are now in the position to prove that $r=1$. It is obvious that $q\neq p$ since
 \begin{equation}\label{eq:30}
 Q_q=\sum_{r=0}^{(q-1)/2}\binom
 q{2r}(kx_1^2/4)^{(q-2r-1)/2}(ly_1^2/4)^r.
 \end{equation}
If $r>1$, then from
 \begin{equation}\label{eq:31}
 Q_{q,q}=\sum_{r=0}^{(q-1)/2}\binom
 q{2r}(kx_q^2/4)^{(q-2r-1)/2}(ly_q^2/4)^r.
\end{equation}
We will prove that $q=3$, $Q_{q,q}=3^t,\; 3|k$. Since
$Q_3=(kx_1^2+3ly_1^2)/4$, hence we get $3|Q_3$ and $3|x_3$. If
$t=1$, from equation \eqref{eq:31} we obtain $12=kx_3^2+3ly_3^2=4ly_3^2+4$. It follows that $ly_3^2=2$ which is impossible. If $t>1$, again from equation \eqref{eq:31} we see that $9|3ly_3^2$. This is also impossible. We have shown that $r=1$. If $m_1>1$ and since $(Q_{m_1}, Q_q)=1$, we get $m_1=5, k=5, l=1$ by Theorem \ref{Thm:3.4}. From $$ Q_{5,q}=\sum_{r=0}^{2}\binom
 5{2r}(5x_q^2/4)^{(5-2r-1)/2}(y_q^2/4)^r$$ and as every prime divisor of $Q_{5,q}$ divides $5$, we obtain $Q_{5,q}=5=Q_5=y_5$. This leads to a contradiction. Thus $m_1=1$, therefore $m=q$ as desired.\\

{\bf Case 2:} Now we assume that $p\nmid Q_{q^r}.$ By Theorem \ref{Thm:3.4} we get
$k=5, l=1, q=5,r=1$. Write $m_1=m_2P^u,$ where $P$ is a prime divisor of $m_1$ and $(m_2,P)=1, u\in \mathbb{N}.$ If $P\nmid Q_{P^u}$, again by Theorem \ref{Thm:3.4} we get $P=5$, contradicting the fact that $(m_1,5)=1$. Hence we have $p|Q_{P^u}$. By what we proved before we know that $m_1=P$, and so $m=5P$. Again it is also  impossible. This completes the proof of (i).\\

The proof of (ii) can be done exactly by the same way. The proof is complete.
 \end{proof}

\begin{theorem}\label{thm:3.12} Let the assumption be as in Theorem \ref{thm:3.11}.\\
(i) If $x=p_1^{n_1}p_2^{n_2}x'$, where $p_1,p_2$ are primes not
dividing $k$, $n_1,n_2\in \mathbb{N}$, and every prime divisor of $x'$ divides $k$, then $\frac{x\sqrt{k}+y\sqrt{l}}{2}=\frac{\varepsilon}{2}$ or $(\frac{\varepsilon}{2})^q$ or
 $(\frac{\varepsilon}{2})^{q^2}$ or $(\frac{\varepsilon}{2})^{5q}$. The last case occurs only if $k=5,\; l=1,$ except for the case $(k,l,x)=(5,1,5^3\cdot3001\cdot158414167964045700001)$. Here $\varepsilon=x_1\sqrt{k}+y_1\sqrt{l}$ is the minimal positive solution of equation \eqref{eq:11}, $q$ an odd prime with $(q,p_1p_2)=1$.\\
 (ii) If $y=p_1^{n_1}p_2^{n_2}y'$, where $p_1,p_2$ are primes not dividing $D$ and $n_1,n_2\in \mathbb{N}$, and every prime divisor of $y'$ divides $l$, then $\frac{x\sqrt{k}+y\sqrt{l}}{2}=\frac{\varepsilon}{2}$ or $(\frac{\varepsilon}{2})^q$ or $(\frac{\varepsilon}{2})^{q^2},$ where $q$ is an odd prime with $(q, p_1p_2)=1$.
 \end{theorem}

\begin{proof} (i) It is enough to prove the result for the case of
$p_1\nmid x_1,\; p_2\nmid x_1$ by Remark \ref{rmk:3.1}. Using Lemma \ref{lem:2.2}, we know that
\begin{equation}\label{eq:32}
 \frac{x\sqrt{k}+y\sqrt{l}}{2}= \left(\frac{x_1\sqrt{k}+y_1\sqrt{l}}{2}\right)^m,
 \end{equation}
 for some odd positive integer $m$. If $m=1$, the result is immediate. Hence we assume $m>1$. It is not difficult to prove that $(m, p_1p_2)=1$. We write $m=m_1q^r$, where $q$ is a prime divisor of $m$, $(m_1,q)=1, r\in \mathbb{N}.$ For the remaining of the proof we consider three cases.\\

 {\bf Case 1:} Assume that $p_1\nmid Q_{q^r}$ and $p_2\nmid Q_{q^r}$. Then every prime divisor of $x_{q^r}=x_1Q_{q^r}$ divides
 $k$ by the assumption. It follows that $q=5,r=1,k=5,l=1$ by Theorem \ref{Thm:3.4}. Let $P$ be a prime divisor of $m_1$, then by
 Theorem \ref{Thm:3.4} we know that it can be at least one of $p_1$ and $p_2$. Without loss of generality, we take $p_1$ such that $p_1|Q_P.$ Since $(Q_P,Q_{5m_2,P})|5m_2$, where $m_1=Pm_2$, then from
 $$p_1^{n_1}p_2^{n_2}x'=x_m=x_PQ_{5m_2,P}.$$
 We see that $Q_{5m_2,P}=p_2^nx''$, where $n=0$ or $n=n_2,x''|x'$. But $(Q_{5m_2,P},P_{5m_2,P})$ is a solution of the Diophantine equation
 $$k_1x^2-l_1y^2=4,k_1=kx_P^2,l_1=ly_P^2.$$
 By Theorem \ref{thm:3.11}, we get $m_2=1$ so $m=5P$ as desired.\\

 {\bf Case 2:} There exists one and only one of $p_1$ and $p_2$ dividing $Q_{q^r}$. Without loss of generality, we take $p_1$ such that $p_1|Q_{q^r}.$ Then $m_1>1$. If $p_1\nmid Q_q$, then $q=5,r=2,k=5,l=1,p_1=3001$ by Theorem \ref{Thm:3.4}. Since $(Q_{m_1},Q_{25})=Q_{(m_1,25}=Q_1=1$, then $p_2|x_{m_1}$ and $p_1\nmid x_{m_1}$. By Theorem \ref{thm:3.11}, we know that $m_1=P$ is an odd prime. So
 $$Q_{25P,P}=p_1^{n_1}x'',x''|x',$$
 and $(Q_{25P,P},P_{25P,P})$ is a solution of Diophantine equation
of
$$k_1x^2-l_1y^2=4,\,k_1=5x_P^2,\,l_1=y_P^2,$$
 contradicting Theorem \ref{thm:3.11}. If $p_1|Q_q$, by what we proved before we know that
 $$Q_{m_2q,q}=p_2^{n_2}x'',x''|x',m=m_2q.$$
But $(Q_{m_2q,q},P_{m_2q,q})$ is a solution of the Diophantine equation
$$k_1x^2-l_1y^2=4,\,k_1=kx_q^2,\,l_1=ly_q^2.$$
 Hence by Theorem \ref{thm:3.11}, we know that $m_2=P$ is an odd prime. We claim that $P=q.$ Otherwise from
$$p_1^{n_1}p_2^{n_2}x'=x_PQ_{Pq,P}$$
and
$$p_1^{n_1}p_2^{n_2}x'=x_qQ_{Pq,q},$$
we see that $Q_{Pq,P}=Q_q$, which is impossible.\\

{\bf Case 3:} Assume that $p_1|Q_{q^r}$ and $p_2|Q_{q^r}$. If $p_1\nmid Q_{q}$ and $p_2\nmid Q_{q}$, then every prime divisor of $x_{q}=x_1Q_{q}$ divides $k$ by the assumption. It follows that $q=5,k=5,l=1$ by Theorem \ref{Thm:3.4}.
 A simple calculation gives $x_{25}=75025=5^2\cdot3001, x_{5^3}=5^3\cdot3001\cdot158414167964045700001$.
 It follows that $r=3,m_1=1.$ Therefore we get $m=125=5^3.$ If one of $p_1$ and $p_2$ divides $Q_q$, without loss of generality we may assume that $p_1|Q_q$, then
$$Q_{m_2q,q}=p_2^nx'',x''|x',n=0 \,or \, n_2.$$
But $(Q_{m_2q,q},P_{m_2q,q})$ is a solution of the Diophantine equation
$$k_1x^2-l_1y^2=4,\,k_1=kx_q^2,\,l_1=ly_q^2.$$
Note that $k_1>5.$ Again by Theorem \ref{thm:3.11}, we see that $m_2=P$ is an odd prime. Again by what we proved, we have $P=q.$ This completes the proof of (i).\\

The proof of (ii) can be done exactly by the same method. The proof is complete.
\end{proof}
%----------------------------------------------------------------%

\section{ Applications}\label{sec:4}

\subsection{The Diophantine equation $kx^2-ly^2=1$} \label{subsec:4.1}

 First we consider the solvability of the Diophantine equation
$kx^2-ly^2=1$, where $kl=D$ is a given positive nonsquare integer.\\

It is well known that Diophantine equation \eqref{eq:8} has positive integer solutions $(x, y)$, for any positive nonsquare integer $D$. But the solvability of Diophantine equations \eqref{eq:2}, \eqref{eq:4}, \eqref{eq:5}, \eqref{eq:9} is not sure. For example, the equation $3x^2-2y^2=1$ is solvable in positive integers $x$, $y$ and the equation $2x^2-3y^2=1$ has no positive integer solutions. O. Perron \cite{perr} proved the following result.
\begin{theorem}{\rm (\cite{perr})}\label{thm:4.1} Let $D\neq 2$ be a positive nonsquare integer. Then there is at least one of the following three Diophantine equations
$$x^2-Dy^2=-1,\,x^2-Dy^2=2,\,x^2-Dy^2=-2$$
which has integer solutions.
\end{theorem}

A more general result was obtained by the first author.
\begin{theorem}{\rm (\cite{yuan3})} \label{thm:4.2} Let $D\neq 2$ be a given positive nonsquare integer with $8\nmid D$.\\
(i) If $2|D$, then the following Diophantine equation
\begin{equation}\label{eq:33}
kx^2-ly^2=1,
\end{equation}
has integer solutions, where $k, l$ are such that $k>1,\; kl=D$.\\
(ii) If $2\nmid D$, then there is one and only one of the following
Diophantine equations
\begin{equation}\label{eq:34}
kx^2-ly^2=1,\,kx^2-ly^2=2
\end{equation}
which has integer solutions, where $k, l$ are such that $k>1,\; kl=D$ for the first equation and $k, l$ are such that $k>0,\; kl=D$ for the second equation.\\
(iii) If $2\nmid D$ and the Diophantine equation $x^2-Dy^2=4$ has
solutions in odd integers $x$ and $y$, then there is at least on $(k, l)$ with $k>1$, $kl=D$ such that the following Diophantine equation
\begin{equation}\label{eq:35}
kx^2-ly^2=4
\end{equation}
has integer solutions.
\end{theorem}

\begin{proof} Let $x_0+y_0\sqrt{D}$ be the fundamental solution of
the Diophantine equation
\begin{equation}\label{eq:36}
x^2-Dy^2=1.
\end{equation}
Then
\begin{equation}\label{eq:37}
(x_0+1)(x_0-1)=Dy_0^2.
\end{equation}
 (i) By the assumption we know that $(x_0-1, x_0+1)=2.$ Thus as $8\nmid D$ and from equation \eqref{eq:37} we get
\begin{equation}\label{eq:38}
x_0+1=2ku^2,\,x_0-1=2lv^2,\,kl=D,2uv=y_0.
\end{equation}
So the Diophantine equation
\begin{equation}\label{eq:39}
kx^2-ly^2=1
\end{equation}
has a solution $(u,v)$. We assume $k>1.$ Otherwise $l=D$ and
$(u,v)$ is a positive integer solution to the Diophantine equation
\eqref{eq:36}, contradicting the fact that $x_0+y_0\sqrt{D}$ is the fundamental solution since $v<y_0$.  Suppose now that there exist positive integers $k_1,l_1,k_2,l_2$ with $k_1l_1=k_2l_2=D,\, k_1>1,\, k_2>1,$ and Diophantine equations $k_1x^2-l_1y^2=1,\; k_2x^2-l_2y^2=1$ are all solvable in positive integers $x$ and $y$. Let $(x_i,y_i)$ be the minimal positive solution of the Diophantine equation $k_ix^2-l_iy^2=1,\, i=1, 2$. Then we have
\begin{equation}\label{eq:40}
k_1x_1^2-l_1y_1^2=k_2x_2^2-l_2y_2^2=1.
\end{equation}
By Lemma \ref{lem:2.4} we have
$$x_0+y_0\sqrt{D} = (x_1\sqrt{k_1}+y_1\sqrt{l_1})^2=(x_2\sqrt{k_2}+y_2\sqrt{l_2})^2.$$
Hence
\begin{equation}\label{eq:41}
k_1x_1^2+l_1y_1^2=k_2x_2^2+l_2y_2^2.
\end{equation}
From equations \eqref{eq:40} and \eqref{eq:41} we obtain
\begin{equation}\label{eq:42}
k_1x_1^2=k_2x_2^2,\,l_1y_1^2=l_2y_2^2.
\end{equation}
It follows that
$(k_1,l_2y_2^2)=(k_1,l_1y_1^2)=1,\,(k_2,l_1y_1^2)=(k_2,l_1y_1^2)=1$.
Therefore we have $(k_1,l_2)=(k_2,l_1)=1.$ Thus as $k_1l_1=k_2l_2=D$, we can easily prove that $k_1=k_2,l_1=l_2,$ as desired. This completes the proof of (i).\\

(ii) If $x_0$ is an odd, then $(x_0-1,x_0+1)=2.$ Thus using the same type of argument as in (i), we will prove that there exists one and
only one pair of positive integers $(k,l)$ with $k>1, kl=D$ such that the Diophantine equation
\begin{equation}\label{eq:43}
kx^2-ly^2=1.
\end{equation}
has positive integer solutions.

If $x_0$ is an even, then $(x_0-1,x_0+1)=1.$ Thus from equation \eqref{eq:37} we get
\begin{equation}\label{eq:44}
x_0+1=ku^2,\,x_0-1=lv^2,\,kl=D,uv=y_0.
\end{equation}
Therefore the Diophantine equation
\begin{equation}\label{eq:45}
kx^2-ly^2=2
\end{equation}
has a solution $(u,v)$. Suppose now that there exist positive integers $k_1, l_1, k_2, l_2$ with $k_1l_1=k_2l_2=D,$ and Diophantine equations $k_1x^2-l_1y^2=2,\; k_2x^2-l_2y^2=2$ are all solvable in positive integers $x$ and $y$. Let $(x_i,y_i)$ be the minimal positive solution of the Diophantine equation $k_ix^2-l_iy^2=2,i=1,2$. Then
\begin{equation}\label{eq:46}
k_1x_1^2-l_1y_1^2=k_2x_2^2-l_2y_2^2=2.
\end{equation}
By Lemma \ref{lem:2.4}, we have
$$x_0+y_0\sqrt{D} = \left(x_1\sqrt{k_1}+y_1\sqrt{l_1}\right)^2/2 = \left(x_2\sqrt{k_2}+y_2\sqrt{l_2}\right)^2/2.$$
Hence we obtain
\begin{equation}\label{eq:47}
k_1x_1^2+l_1y_1^2=k_2x_2^2+l_2y_2^2.
\end{equation}
From equations \eqref{eq:46} and \eqref{eq:47} we get
\begin{equation}\label{eq:48}
k_1x_1^2=k_2x_2^2,\,l_1y_1^2=l_2y_2^2.
\end{equation}
It follows that
$(k_1,l_2y_2^2)=(k_1,l_1y_1^2)=1,\,(k_2,l_1y_1^2)=(k_2,l_1y_1^2)=1$,
and so $(k_1,l_2)=(k_2,l_1)=1.$ Thus from $k_1l_1=k_2l_2=D$, we can
easily prove that $k_1=k_2,l_1=l_2,$ as desired. Finally, suppose
that there exist positive integers $k_1,l_1,k_2,l_2$ with
$k_1l_1=k_2l_2=D,$ and Diophantine equations
$k_1x^2-l_1y^2=1,k_2x^2-l_2y^2=2$ are all solvable in positive
integers $x$ and $y$. Let $(x_1,y_1)$ and $(x_2,y_2)$ be the minimal positive solution of the Diophantine equation $k_1x^2-l_1y^2=1$ and
$k_2x^2-l_2y^2=2$ respectively. Then by Lemma 2.4 we have
$$x_0+y_0\sqrt{D} = \left(x_1\sqrt{k_1}+y_1\sqrt{l_1}\right)^2 = \left(x_2\sqrt{k_2}+y_2\sqrt{l_2}\right)^2/2.$$
Hence we get
\begin{equation}\label{eq:49}
2x_1y_1=x_2y_2.
\end{equation}
From equation \eqref{eq:49}, we obtain $2|x_2y_2$, which is impossible since $2\nmid k_2l_2$ and $k_2x_2^2-l_2y_2^2=2.$ This completes the proof of (ii).\\

The proof of (iii) can be done exactly by the same method. This completes the proof of Theorem \ref{thm:4.2}.
\end{proof}

 Similarly, we will prove
\begin{theorem}\label{thm:4.3} Let $D$ be a given positive
nonsquare integer with $8|D$. Then there is at least one pair $(k,l)$ with $k>1, kl=D$ such that the Diophantine equation
\begin{equation}\label{eq:50}
kx^2-ly^2=1,
\end{equation}
has integer solutions.
\end{theorem}

%--------------------------------------------------------------------%

\subsection{On Sierpinski's conjecture}\label{subsec:4.2}
The second application of St$\ddot{o}$rmer theorem deals with Sierpinski's conjecture. An integer of the form
$T_n=n(n+1)/2, n\in\mathbb{N}$ is called triangular number. In \cite[D23]{guy04}, it is stated that Sierpinski asked whether or not there exist four (distinct) triangular numbers in geometric progression. Szymiczek \cite{sz72} conjectured that the answer is negative. The problem of finding three such triangular numbers is readily reduced to finding solutions to a Pell equation. (An old result of G$\acute{e}$rardin \cite{ge14}. See also \cite{Di20}, \cite{sz63}). This implies that there  are infinitely many such triples. The smallest triple is $(T_1,\,T_3,\,T_8)$. In fact, a simple calculation shows that if $T_n=m^2$ then $T_n, \,
  T_{n+2m}=m(2n+3m+1), \, T_{3n+4m+1}=(2n+3m+1)^2$ form a geometric
  progression.

Recently M.Bennett \cite{be05} proved the following result.
  \begin{theorem}{\rm (\cite{be05})} \label{thm:4.4}
  There do not exist four distinct triangular numbers in geometric progression with the ratio being  a positive integer.
  \end{theorem}

Chen and Fang \cite{cf07} extended Bennett's result to the rational ratio and proved the following theorem.
  \begin{theorem}{\rm (\cite{cf07})} \label{thm:4.5} There do not exist four distinct  triangular numbers in geometric progression.
  \end{theorem}

  Using the theory of Pell equations and a result of Bilu-Hanrot-Voutier \cite{Bilu-Hanrot-Voutier:2001}
  on primitive divisors of Lucas and Lehmer numbers, Yang-He \cite{yh07} and Yang \cite{Yang:2008} claimed
  that there is no geometric progression which contains four distinct triangular numbers. But their proof
  is under the assumption that the geometric progression has an integral common ratio. Fang \cite{fa07},
  using only the St$\ddot{o}$rmer theorem on Pell's equation, showed that no geometric progression contains
  four distinct triangular
   numbers.\\

   Note that if $T_n=n(n+1)/2, n\in\mathbb{N}$ is a triangular number,
   then $8T_n=m^2-1$, where $m=2n+1$. One can notice that Sierpinski problem is
    equivalent to ask whether or not there exist four distinct integers of the form $m^2-1$
    in geometric progression. In this paper, we consider the question whether or not there exist four
    distinct integers of the form $Dm^2\pm C$  with $D, m\in\mathbb{N}, C=1,2,4$ in geometric progression.
    Recently, the first two authors used only Theorems \ref{Thm:3.2}, \ref{Thm:3.3}, \ref{Thm:3.4} and a
     result of Lehmer numbers to prove the following results.

\begin{theorem}{\rm (\cite{luoy1})}\label{thm:4.6} Let $D$ be a positive integer, then no geometric progression contains four distinct integers of the form $Dm^2\pm 1,m\in\mathbb{N}$.
\end{theorem}
\begin{theorem}{\rm (\cite{luoy1})}\label{thm:4.7} Let $D$ be a positive integer, then there do not exist four distinct integers of the form $Dm^2\pm
2,m\in\mathbb{N}$ in geometric progression.
\end{theorem}
\begin{theorem}{\rm (\cite{luoy1})}\label{thm:4.8} Let $D$ be a positive integer, then there do not exist four distinct integers of the form $Dm^2\pm
4,m\in\mathbb{N}$ in geometric progression.
\end{theorem}

%-----------------------------------------------------------------%

\subsection{On Ma's conjecture}\label{subsec:4.3}

The third application is related to Ma's conjecture. Let $\mathbb{N}_0$ be the set of all nonnegative integers. In 1992, S. L. Ma \cite{ma01} presented the following conjecture.
 \begin{conjecture}\label{conj:Ma} Let $p$ an odd prime and $b,t,r\in\mathbb{N}.$ Then\\
 (A) $Y=2^{2b}p^{2t}-2^{2b}p^{t+r}+1$ is a square if and only if
 $t=r$, that is, if and only if $Y=1.$\\
 (B) $Z=2^{2b+2}p^{2t}-2^{b+2}p^{t+r}+1$ is a square if and only if
  $p=5,b=3,t=1$, and $r=2$, that is, if and only if $Z=2401.$
 \end{conjecture}

  Moreover, Ma proved that conjecture (A) and (B) implies McFarland's conjecture on Abelian difference sets with multiplier $-1$. See also \cite{jun01}, Conjecture 13.9. In 1993, using Lemma \ref{lem:2.4} and Theorem \ref{Thm:3.1}, Zhenfu Cao \cite{cao01} proved the following result.
  \begin{theorem}{\rm (\cite{cao01})} \label{thm:4.9} If $t_i>r_i\; (i=1,2,\cdots,s),\; a\geq b,$ then the
  Diophantine equation
  $$x^2=2^{2a}p_1^{2t_1}\cdots p_s^{2t_s}-2^{a+b}p_1^{t_1+r_1}\cdots p_s^{t_s+r_s}+1$$
 has no solutions with $x,a,b,t_i,r_i\in\mathbb{N},\; i=1,2,\cdots,s,$
 where the $p_i$ are prime numbers.
 \end{theorem}

  For $s=1$, Theorem \ref{thm:4.9} shows that Ma's conjecture (A) is true. In 1996, M. Le and Q. Xiang \cite{lexiang} proved also that Theorem \ref{thm:4.9} holds for the special case $s=1.$ Later,
  Y. D. Guo \cite{guo01} obtained the following theorem.
  \begin{theorem}{\rm (\cite{guo01})} \label{thm:4.10} If $k$  is an odd integer $>1$ and $2n>m>n,$ then the Diophantine equation
  $$x^2=2^{2a}k^{2m}-2^{2a}k^{m+n}+1$$
 has no solutions with $x,k,a,b,m,n\in\mathbb{N}.$
 \end{theorem}

 In 2000, using Lemma \ref{lem:2.4} and Theorem \ref{Thm:3.1}, Z.Cao and A.Grytczuk get a more general result. In fact, they proved the following result.
\begin{theorem}{\rm (\cite{cao02})} \label{thm:4.11} If $t_i>r_i\; (i=1,2,\cdots,s),$ (i) $ a>b$  or (ii) $a=b$ and $y$ is odd, then the Diophantine equation
  $$x^2=2^{2a}k_1^{2t_1}\cdots k_s^{2t_s}y^2-2^{a+b}k_1^{t_1+r_1}\cdots k_s^{t_s+r_s}\delta+1,\delta \in \{-1,1\}$$
 has only the solution
 $$x=2^{a+b-1}k_1^{t_1+r_1}\cdots k_s^{t_s+r_s}-\delta,\; y=2^{b-1}k_1^{r_1}\cdots k_s^{r_s}$$
 with $x,y,a,b,k_i,t_i,r_i\in\mathbb{N},i=1,2,\cdots,s.$
 \end{theorem}

 In 2002, using Lemma \ref{lem:2.4} and Theorems \ref{Thm:3.1}, \ref{Thm:3.5}, X. Dong and Z. Cao \cite{dongcao} showed the following result.
\begin{theorem}{\rm (\cite{dongcao})}\label{thm:4.12} Let $t_i>r_i\; (i=1,2,\cdots,s),\; a\geq b$ and $k_1,k_2,\cdots k_s>1$.\\
(i) Assume that $p=2.$ Then the Diophantine equation
\begin{equation}\label{eq:51}
x^2=p^{2a}k_1^{2t_1}\cdots k_s^{2t_s}y^2-p^{a+b}k_1^{t_1+r_1}\cdots
k_s^{t_s+r_s}\delta+1,\delta \in \{-1,1\},
\end{equation}
where
$$x,y,a,b,k_i,t_i\in\mathbb{N},\; i=1,2,\cdots,s,\; r_i\in \mathbb{N}_0,\; i=1,2,\cdots,s,$$
and $p$ is a prime, has only the solution
$$x=2^{a+b-1}k_1^{t_1+r_1}\cdots
k_s^{t_s+r_s}-\delta,\; y=2^{b-1}k_1^{r_1}\cdots
 k_s^{r_s}.$$\\

(ii) Assume that $p>2$. Then the only solutions to equation \eqref{eq:51} are:

(A) If there is a $j$ $(1\leq j \leq s)$ such that $2|k_j^{r_j}$,  then
$$x=\frac{1}{2}p^{a+b}k_1^{t_1+r_1}\cdots
k_s^{t_s+r_s}-\delta,y=\frac{1}{2}p^bk_1^{r_1}\cdots k_s^{r_s}.$$

(B) If there is a $j$ $(1\leq j \leq s)$ such that $4|k_j^{r_j}$ and $p^b=\frac{1}{8}k_1^{r_1}\cdots k_s^{r_s}-\delta,$ then
$$x=2p^{2b}-1,\; y=\frac{1}{4}k_1^{r_1}\cdots k_s^{r_s},\; a=b.$$

(C) If there is a $j$ $(1\leq j \leq s)$ such that
$$p^b=3^{r_j}-2=3^{-2r_j}k_1^{t_1+r_1}\cdots
k_s^{t_s+r_s}+1,\; 2\nmid r_j>1$$ and $3|k_j$, then
$$x=\frac{1}{2}(p^{b}+1)((p^b+1)^2-3),\; y=\frac{1}{2}3^{-r_j}k_1^{r_1}\cdots k_s^{r_s},a=b,\; \delta=-1.$$
 \end{theorem}
Further, we consider the Diophantine equation
\begin{equation}\label{eq:52}
x^2=p^{2a}k_1^{2t_1}\cdots k_s^{2t_s}y^2-p^{a+b}k_1^{t_1+r_1}\cdots
k_s^{t_s+r_s}\delta+1,\; \delta \in \{-2,-4,2,4\},
\end{equation}
where
$$x,y,a,b,k_i,t_i\in\mathbb{N},\; i=1,2,\cdots,s,\; r_i\in
\mathbb{N}_0,\; i=1,2,\cdots,s,\; 2\nmid y,$$ and $p$ is an odd prime. By Lemma \ref{lem:2.4} and Theorems \ref{Thm:3.1}, \ref{Thm:3.5}, \ref{thm:4.2}, we get the following result.
\begin{theorem} \label{thm:4.13} Let $t_i>r_i\; (i=1,2,\cdots,s),\; a\geq b$ and $k_1,k_2,\cdots k_s>1$ be odd.\\

(i) Assume that  $\delta=-2$ or $\delta=2$,  then the only solutions to equation \eqref{eq:52} are:

(A) If $p|y$, then
$$x=p^{a+b}k_1^{t_1+r_1}\cdots k_s^{t_s+r_s}-\frac{\delta}{2},\; y=p^bk_1^{r_1}\cdots k_s^{r_s}.$$

(B) If $\delta=2$ and
$$p^b=3^{t}+2=2\cdot3^{-2t}k_1^{t_1+r_1}\cdots
k_s^{t_s+r_s}-1,\; 2\nmid t,$$
then
$$x=\frac{1}{2}(p^{b}-1)((p^b-1)^2-3),\; y=3^{-t}k_1^{r_1}\cdots k_s^{r_s},\; a=b.$$

(C) If $\delta=-2$ and
$$p^b=3^{t}-2=2\cdot3^{-2t}k_1^{t_1+r_1}\cdots
k_s^{t_s+r_s}+1,\; 2|t,$$  then
$$x=\frac{1}{2}(p^{b}+1)((p^b+1)^2-3),\; y=3^{-t}k_1^{r_1}\cdots k_s^{r_s},\; a=b.$$

(ii) Assume that $\delta=-4$ or $\delta=4$, then the Diophantine
equation \eqref{eq:52} has no positive integer solutions.
\end{theorem}

\begin{proof} Let $$l=p^{a-b}k_1^{t_1-r_1}\cdots
k_s^{t_s-r_s},\; D=l(ly^2-\delta),\; z=p^bk_1^{r_1}\cdots k_s^{r_s}.$$
From equation \eqref{eq:52}, one can see that $(x,z)$ is a solution of the Diophantine equation
\begin{equation}\label{eq:53}
X^2-DY^2=1.
\end{equation}
(i) Since $(y,1)$ is the minimal positive solution of Diophantine
equation
$$lX^2-(ly^2-\delta)Y^2=\delta,$$
hence
$(y\sqrt{l}+\sqrt{l^2-\delta})^2/2 = ly^2-(\delta/2)+y\sqrt{l(ly^2-\delta)}$
is the fundamental solution of equation \eqref{eq:53}. By Theorem \ref{Thm:3.1} and Theorem \ref{Thm:3.5}, we have
\begin{equation}\label{eq:54}
x=ly^2-(\delta/2),\,z=y,
\end{equation}
or
\begin{equation}\label{eq:55}
x=(ly^2-(\delta/2))^2+l(ly^2-\delta)y^2,\,z=2y(ly^2-(\delta/2)),
\end{equation}
or
\begin{equation}\label{eq:56}
x=(ly^2-\frac{\delta}{2})^3+3(ly^2-\frac{\delta}{2})l(ly^2-\delta)y^2,\, z/y=(2ly^2-\delta)^2-1=3^tp^b,a=b.
\end{equation}
It is easy to see that equation \eqref{eq:55} is impossible since $z$ is an odd integer. From equation \eqref{eq:54}, we have $p|y$ and
$$x=p^{a+b}k_1^{t_1+r_1}\cdots
k_s^{t_s+r_s}-\frac{\delta}{2},y=p^bk_1^{r_1}\cdots k_s^{r_s}.$$
If equation \eqref{eq:56} is true, then $3^ty=k_1^{r_1}\cdots k_s^{r_s}$. It follows that $y=3^{-t}k_1^{r_1}\cdots k_s^{r_s}$. If $\delta=2$, again from equation \eqref{eq:56}, we obtain $2ly^2-3=3^t,2ly^2-1=p^b$. We get $2\nmid t$ by considerations by modulo $4$. Therefore we have
$$p^b=3^{t}+2=2\cdot3^{-2t}k_1^{t_1+r_1}\cdots
k_s^{t_s+r_s}-1,\; x=\frac{1}{2}(p^{b}-1)((p^b-1)^2-3).$$
If $\delta=-2$, also from equation \eqref{eq:56}, we see that $2ly^2+3=3^t,\; 2ly^2+1=p^b$. We
get $2|t$ by considerations by modulo $4$. Therefore we have
$$p^b=3^{t}-2=2\cdot3^{-2t}k_1^{t_1+r_1}\cdots
k_s^{t_s+r_s}+1,\; x=\frac{1}{2}(p^{b}+1)((p^b+1)^2-3).$$
This proves (i) of Theorem.\\

(ii) Since $(y,1)$ is a positive integer solution of Diophantine
equation
$$lX^2-(ly^2-\delta)Y^2=\delta,$$
hence $$u\sqrt{l}+v\sqrt{ly^2-\delta}=((y\sqrt{l}+v\sqrt{ly^2-\delta})/2)^3$$ is a solution of Diophantine equation $lX^2-(ly^2-\delta)Y^2=\pm 1$. On the other hand, if \eqref{eq:52} is true, note that $2\nmid Dz^2$ and from \eqref{eq:53}, we get
$$x+1=D_1z_1^2,x-1=D_2z_2^2,D_1D_2=D,z=z_1z_2.$$
Thus $(z_1,z_2)$ is a solution of $D_1X^2-D_2Y^2=2$, contradicting
Theorem \ref{thm:4.2}. This completes the proof of Theorem \ref{thm:4.13}.
\end{proof}
We now consider the Diophantine equation
\begin{equation}\label{eq:57}
x^2=p^{2a}k_1^{2t_1}\cdots k_s^{2t_s}y^2-p^{a+b}k_1^{t_1+r_1}\cdots
k_s^{t_s+r_s}\delta+4,\; \delta \in \{-4,4\},
\end{equation}
where $$x,y,a,b,k_i,t_i\in\mathbb{N},\; i=1,2,\cdots,s,\; r_i\in
\mathbb{N}_0,\; i=1,2,\cdots,s,\; 2\nmid y,$$
and $p$ is an odd prime. Using Lemma \ref{lem:2.4} and Theorems \ref{Thm:3.4}, \ref{thm:3.9}, we get the following result.
\begin{theorem}\label{thm:4.14} Let $t_i>r_i\; (i=1,2,\cdots,s),\; a\geq b$ and $k_1,k_2,\cdots k_s>1$ be odd.
Then except for $(x,y,p,s,r_1,t_1,k_1,a,b,\delta) = (123,1,11,1,1,2,5,1,1,4)$, the only solutions to equation \eqref{eq:57} are given by:

(A) If $p|y$, then
$$x=p^{a+b}k_1^{t_1+r_1}\cdots
k_s^{t_s+r_s}-\frac{\delta}{2},\; y=p^bk_1^{r_1}\cdots k_s^{r_s}.$$

(B) If $p\nmid y$ and $p^b=k_1^{t_1+r_1}\cdots
k_s^{t_s+r_s}-\frac{\delta}{2}$, then
$$x=p^{2b}-2,\; y=k_1^{r_1}\cdots k_s^{r_s},\; a=b.$$
\end{theorem}

\begin{proof} Let $$l=p^{a-b}k_1^{t_1-r_1}\cdots
k_s^{t_s-r_s},\; D=l(ly^2-\delta),\; z=p^bk_1^{r_1}\cdots k_s^{r_s}.$$
From equation \eqref{eq:57}, one sees that $(x,z)$ is a solution of Diophantine equation
\begin{equation}\label{eq:58}
X^2-DY^2=4.
\end{equation}
Since $(y,1)$ is the minimal solution of Diophantine equation
$$lX^2-(ly^2-\delta)Y^2=\delta,$$
hence
$(y\sqrt{l}+\sqrt{l^2-\delta})^2/2=ly^2-(\delta/2)+y\sqrt{l(ly^2-\delta)}$
 is the fundamental solution of equation \eqref{eq:58} by Lemma \ref{lem:2.4}. By Theorems \ref{Thm:3.4}, \ref{thm:3.9} and Proposition~\ref{prop:Lehmer}, we have
\begin{equation}\label{eq:59}
x=ly^2-(\delta/2),\,z=y,
\end{equation}
or
\begin{equation}\label{eq:60}
x=((ly^2-(\delta/2))^2+l(ly^2-\delta)y^2)/2,\,z=y(ly^2-(\delta/2)),a=b,
\end{equation}
or
\begin{equation}\label{eq:61}
x=123,D=5,p^bk_1^{r_1}\cdots k_s^{r_s}=55,a=b.
\end{equation}

From equation \eqref{eq:59}, we obtain $p|y$ and
$$x=p^{a+b}k_1^{t_1+r_1}\cdots
k_s^{t_s+r_s}-\frac{\delta}{2},\; y=p^bk_1^{r_1}\cdots k_s^{r_s}.$$
From equation \eqref{eq:60}, one can see that $p\nmid y$ and
$$p^b=k_1^{t_1+r_1}\cdots k_s^{t_s+r_s}-\frac{\delta}{2},\; x=p^{2b}-2,\; y=k_1^{r_1}\cdots k_s^{r_s},\; a=b.$$
If equation \eqref{eq:61} is true, we can easily see that
$$y=1, p=11, s=1, r_1=1, t_1=2, k_1=5, a=b=1, \delta=4.$$
This completes the proof of Theorem \ref{thm:4.14}.
\end{proof}

%--------------------------------------------------------------%

\subsection{More Diophantine equations}\label{subsec:4.4}

Finally, we discuss the solvability of some Diophantine equations.
Using Theorem \ref{Thm:3.2}, Qi Sun and the first author proved the following results.
\begin{theorem} {\rm (\cite{suny})}\label{thm:4.15} The equation
$$\frac{ax^n-1}{ax-1}=y^2,2\nmid n,$$
holds for some positive integers $a, x, y,$ and $n$ with $x>1, n>1$ if and only if $$x=3, n=2m+1, \; \mbox{and}\; a=\frac{3^{m-1}+1}{4},\; 2|m.$$
 \end{theorem}

 \begin{theorem} {\rm (\cite{suny})}\label{thm:4.16} The equation
$$\frac{ax^n+1}{ax+1}=y^2,2\nmid n,$$
holds for some positive integers $a, x, y$, and $n$ with $x>1, n>1$ if and only if $$x=3,n=2m+1,\; \mbox{and}\; a=\frac{3^{m-1}-1}{4},\;
2\nmid m.$$
 \end{theorem}

If $a=1$, then Theorem \ref{thm:4.15} becomes the famous result of Ljunggren.

\begin{theorem} {\rm (Ljunggren \cite{ljun})}\label{thm:4.17} The equation
$$\frac{x^n-1}{x-1}=y^2,2\nmid n,$$
holds for some positive integers $x,y$, and $n$ with $x>1, n>1$ if and only if $$x=3,\, y=11, \,and\,\, n=5.$$
 \end{theorem}

 Using Theorems \ref{Thm:3.3} and \ref{thm:4.4}, the second author obtained the following results.
\begin{theorem} {\rm (\cite{luo})}\label{thm:4.18}
(A) The equation
$$\frac{ax^n-2}{ax-2}=y^2,2\nmid n,\,2\nmid a$$
holds for some positive integers $a,x,y$ and $n$ with $x>1, n>1$ if
and only if $$x=3,n=2m+1, \; \mbox{and}\; a=\frac{3^{m-1}+1}{2},\;
2\nmid m.$$\\
(B) The equation
$$\frac{ax^n+2}{ax+2}=y^2,2\nmid n,\,2\nmid a$$
holds for some positive integers $a,x,y$, and $n$ with $x>1, n>1$ if and only if $$x=3,n=2m+1, \; \mbox{and}\; a=\frac{3^{m-1}-1}{2},\; 2|m.$$
 \end{theorem}

\begin{theorem} {\rm (\cite{luo})}\label{thm:4.19}
(A) The equation
$$\frac{ax^n-4}{ax-4}=y^2,2\nmid n,\,2\nmid a$$
holds for some positive integers $a,x,y$, and $n$ with $x>1, n>1$ if and only if $$x=5,n=3, \,and\,\, a=1.$$\\
(B) The equation
$$\frac{ax^n+4}{ax+4}=y^2,2\nmid n,\,2\nmid a$$
has no solutions in positive integers $a,x,y$, and $n$ with $x>1,
n>1.$
 \end{theorem}
In 1990, Zhenfu Cao extended Theorems \ref{thm:4.15} and \ref{thm:4.16}.

\begin{theorem} {\rm (\cite{cao03})}\label{thm:4.20}
(A) The equation
$$\frac{ax^m-1}{abx-1}=by^2,2\nmid m,$$
holds for some positive integers $a,b,x,y$, and $m$ with $x>1, m>1$ if and only if $$x=3,m=4s+1,\; a=\frac{3^{2s-1}+1}{4},\; \mbox{and}\; y=3^{2s}+2.$$\\
(B) The equation
$$\frac{ax^m+1}{abx+1}=by^2,2\nmid m,$$
holds for some positive integers $a,b,x,y$, and $m$ with $x>1, m>1$ if and only if $$x=3,m=4s+3, \; a=\frac{3^{2s}-1}{4},\; \mbox{and}\; y=3^{2s+1}-2.$$
\end{theorem}

 The proof of Theorem \ref{thm:4.20} can be given by Theorem \ref{Thm:3.2} and Lemma \ref{lem:2.3}. Applying a method of Lebesgue \cite{le}, M. Filasetu, F. Luca, P. St\^{a}nic\^{a} and R. G. Underwood \cite{flsu01} determined the Galois groups associated with some polynomials constructed using circulant matrices. In the process of determining the Galois groups, the irreducibility of the trinomial $x^{2p}+x^p+m^p$ was established, where $p$ represents an odd prime and $m$ an integer $\geq2$. In paper \cite{flsu02}, they discussed the irreducibility of the more general polynomial $x^{2p}+bx^p+c$  where $b$ and $c$ are nonzero integers. Moreover, they established the irreducibility of the more specific trinomial $x^{2p}+x^p+m^p$ as a consequence of the following Diophantine
 result.
 \begin{theorem} {\rm (\cite{flsu02})}\label{thm:4.21}
The equation
\begin{equation}\label{eq:62}
\frac{ax^{n+2l}-1}{ax^n-1}=y^2,
\end{equation}
 holds for some positive integers $a,x,n$, and $l$ with $x>1$ and some rational number $y$ if and only if
$$2|l,\; a=\frac{3^{l-1}+1}{4},\; x=3, n=1, \; \mbox{and}\; y=\pm (3^l+2).$$
 \end{theorem}

 We can give a proof of Theorem \ref{thm:4.21} by Theorem \ref{Thm:3.2}. A solution to the equation \eqref{eq:62} implies that there exist positive integers $u$ and $v$ satisfying
 $$ax^n-1=du^2\; \mbox{and}\; ax^{n+2l}-1=dv^2$$
where $d$ is a positive squarefree integer dividing
$\gcd(ax^{n+2l}-1,ax^n-1)$. Then we have the equation
$$ax^n(x^l)^2-dv^2=1.$$
It follows that $(x^l,v)$ is a solution of the Diophantine equation
\begin{equation}\label{eq:63}
ax^nX^2-dY^2=1.
\end{equation}
As $ax^n=du^2+1$, then $(1,u)$ is the minimal positive solution of the equation \eqref{eq:63}. Thus by Theorem \ref{Thm:3.2}, we have
$$x^l=3^s,\,\,3^s+3=4ax^n,$$
which implies that
$$x=3,\,\,n=1,\,\,l=s,\,\,a=\frac{3^{l-1}+1}{4},\,\,2|l.$$
But $$\frac{3a3^{2l}-1}{3a-1}=3^{2l}+4\times 3^l+4=(3^l+2)^2.$$
Therefore
$$y=\pm (3^l+2).$$
Using Theorem \ref{Thm:3.2}, we also get the following result.
\begin{theorem}\label{thm:4.22} The equation
\begin{equation}\label{eq:64}
\frac{ax^{n+2l}+1}{ax^n+1}=y^2
\end{equation}
 holds for some positive integers $a,x,n$, and $l$ with $x>1$ and some rational number $y$ if and only if
$$2\nmid l,\,\,a=\frac{3^{l-1}-1}{4},\,\,x=3,n=1 \,\,and\,\, y=\pm (3^l-2).$$
 \end{theorem}
 Similarly, using Theorems \ref{Thm:3.3} and \ref{Thm:3.4} we can obtain the following theorems.
 \begin{theorem} \label{thm:4.23}
 (A) The Diophantine equation
\begin{equation}\label{eq:65}
\frac{ax^{n+2l}-2}{ax^n-2}=y^2,\,\,2\nmid a
\end{equation}
 holds for some positive integers $a,x,n$, and $l$ with $x>1$ and some rational number $y$ if and only if
$$2\nmid l,\,\,a=\frac{3^{l-1}+1}{2},\,\,x=3,n=1, \,\,and\,\, y=\pm
(3^l+2).$$\\
(B) The Diophantine equation
\begin{equation}\label{eq:66}
\frac{ax^{n+2l}+2}{ax^n+2}=y^2,\,\,2\nmid a
\end{equation}
 holds for some positive integers $a,x,n$ and $l$ with $x>1$ and some rational number $y$ if and only if
$$2|l,\,\,a=\frac{3^{l-1}-1}{2},\,\,x=3, n=1, \,\,and\,\, y=\pm (3^l-2).$$
 \end{theorem}

 \begin{theorem}\label{thm:4.24}
 (A) The Diophantine equation
\begin{equation}\label{eq:67}
\frac{ax^{n+2l}-4}{ax^n-4}=y^2,\,\,2\nmid a
\end{equation}
 holds for some positive integers $a,x,n$, and $l$ with $x>1$ and some rational number $y$ if and only if
$$l=n=a=1,\,\,x=5, \,\,and\,\, y=\pm 11.$$\\
(B) Consider the Diophantine equation
\begin{equation}\label{eq:68}
\frac{ax^{n+2l}+4}{ax^n+4}=y^2,\;2\nmid a.
\end{equation}
 There are no solutions to equation \eqref{eq:68} in positive integers $a, x, n,$ and $l$ with $x>1$ and some rational number $y.$
 \end{theorem}
More generally, combining Theorems \ref{thm:4.2} and \ref{thm:4.3}, we can obtain the following theorem.
\begin{theorem} The equation
\begin{equation}\label{eq:69}
\frac{ax^{n+2l}+c}{abt^2x^n+c}=by^2, c=\pm 1,\,\pm 2,\,\pm
4,\,(a,c)=1
\end{equation}
 holds for some integers $a,b,x,y,n,t$, and $l$ with $x>1,n>0,t>0,a>0$ if and only if
$$b=t=1\,\,\,or\,\,\,b=|y|=1$$
except for
\begin{eqnarray}
\notag & (a,b,x,y,n,t,l,c)=(1,-1,2,\pm 3,1,1,1,1),(1,-1,2,\pm 3,2,1,1,2),\\
\notag  & (1,-1,2,\pm 3,3,1,,1,4).
\end{eqnarray}
The solutions to equation \eqref{eq:69} are given by Theorems \ref{thm:4.21} - \ref{thm:4.24} when $b=t=1$. The solutions to equation \eqref{eq:69} are given by $x^l=t$ when $b=|y|=1.$
 \end{theorem}

\begin{proof} For the sake of completeness in this proof, we will give details for the first four values of $c$, even we have to repeat the method used.

\underline{\bf Case 1:} $c=-1.$ \\
 If $b>0$, then from equation \eqref{eq:69} we have
 \begin{equation}\label{eq:70}
 ax^n(x^l)^2-b(abt^2x^n-1)y^2=1.
 \end{equation}
It follows that $(x^l,|y|)$ is a solution of the Diophantine equation
\begin{equation}\label{eq:71}
 ax^nX^2-b(abt^2x^n-1)Y^2=1.
 \end{equation}
It is easy to see that $(t,1)$ is the minimal positive solution of
the Diophantine equation
\begin{equation}\label{eq:72}
 abx^nX^2-(abt^2x^n-1)Y^2=1.
 \end{equation}
By Theorem \ref{thm:4.2}. we get
$$ax^n=abx^n$$
We deduce that $b=1.$ Using Theorem \ref{Thm:3.2}, we obtain
$$x^l=t,|y|=1$$
or
$$x^l=3^st,\,\,3\nmid t,\,\,3^s+3=4ax^nt^2.$$
It follows that
$$l=s,\,\,x=3t_1,\,\,t=t_1^s,\,\,n=1.$$
Then we have
$$3^{s-1}+1=4at_1^{2s+1}.$$
Thus we obtain $t_1=1$ and $t=1.$\\

If $b<0$, then from equation \eqref{eq:69} we have
 \begin{equation}\label{eq:73}
 ax^n(x^l)^2-b(abt^2x^n-1)y^2=1.
 \end{equation}
It follows that $(x^l,|y|)$ is a solution of the Diophantine equation
\begin{equation}\label{eq:74}
 ax^nX^2-b(abt^2x^n-1)Y^2=1.
 \end{equation}
One can see that $(1,t)$ is the minimal positive solution of
the Diophantine equation
\begin{equation}\label{eq:75}
 (-abt^2x^n+1)X^2-(-abx^n)Y^2=1.
 \end{equation}
By Theorem \ref{thm:4.2} we obtain
$$b(abt^2x^n-1)=-abx^n,$$
which is impossible.\\

\underline{\bf Case 2:} $c=1.$ \\

 If $b>0$, then equation \eqref{eq:69} gives
 \begin{equation}\label{eq:76}
 b(abt^2x^n+1)y^2-ax^n(x^l)^2=1.
 \end{equation}
This implies that $(|y|,x^l)$ is a solution of the Diophantine equation
\begin{equation}\label{eq:77}
 b(abt^2x^n+1)X^2-ax^nY^2=1.
 \end{equation}
Again here $(1,t)$ is the minimal positive solution of the Diophantine equation
\begin{equation}\label{eq:78}
 (abt^2x^n+1)X^2-abx^nY^2=1.
 \end{equation}
By Theorem \ref{thm:4.2}, we get
$$ax^n=abx^n.$$
We deduce that $b=1.$ Again by Theorem \ref{Thm:3.2}, we obtain
$$x^l=t,|y|=1$$
or
$$x^l=3^st,\,\,3\nmid t,\,\,3^s-3=4ax^nt^2.$$
Therefore we have
$$l=s,\,\,x=3t_1,\,\,t=t_1^s,\,\,n=1,$$
and
$$3^{s-1}-1=4at_1^{2s+1}.$$
Thus we get $t_1=1$ and $t=1.$\\

If $b<0$, then we deduce from equation \eqref{eq:69}
 \begin{equation}\label{eq:79}
 b(abt^2x^n+1)y^2-ax^n(x^l)^2=1.
 \end{equation}
So $(|y|,x^l)$ is a solution of the Diophantine equation
\begin{equation}\label{eq:80}
 b(abt^2x^n+1)X^2-ax^nY^2=1.
 \end{equation}
Also $(t,1)$ is the minimal positive solution of the Diophantine equation
\begin{equation}\label{eq:81}
 (-abx^n)X^2-(-abt^2x^n-1)Y^2=1.
 \end{equation}
If $b(abt^2x^n+1)>1$, then we use Theorem \ref{thm:4.2} to obtain
$$b(abt^2x^n+1)=-abx^n.$$
This is impossible. Hence we have $b(abt^2x^n+1)=1$ so
$$b=-1,\,\,a=1,\,\,t=1,\,\,n=1,\,\,x=2.$$
Then we get
$$y^2-2^{2l+1}=1.$$
Thus $(|y|,2^l)$ is a solution of the Diophantine equation
$$X^2-2Y^2=1.$$
Therefore we have $y=\pm 3,l=1$ by Theorem \ref{Thm:3.1}.\\

\underline{\bf Case 3:} $c=-2.$ First, we assume that $x$ is an odd integer. \\

 If $b>0$, from equation \eqref{eq:69} we have
 \begin{equation}\label{eq:82}
 ax^n(x^l)^2-b(abt^2x^n-2)y^2=2.
 \end{equation}
Thus $(x^l,|y|)$ is a solution of the Diophantine equation
\begin{equation}\label{eq:83}
 ax^nX^2-b(abt^2x^n-2)Y^2=2.
 \end{equation}
It is obvious to see that $(t,1)$ is the minimal positive solution of the Diophantine equation
\begin{equation}\label{eq:84}
 abx^nX^2-(abt^2x^n-2)Y^2=2.
 \end{equation}
Using Theorem \ref{thm:4.2}, we have
$$ax^n=abx^n,$$
so $b=1.$ Theorem \ref{Thm:3.3} implies
$$x^l=t,|y|=1$$
or
$$x^l=3^st,\,\,3\nmid t,\,\,3^s+3=2ax^nt^2,$$
then
$$l=s,\,\,x=3t_1,\,\,t=t_1^s,\,\,n=1.$$
So we have
$$3^{s-1}+1=2at_1^{2s+1}.$$
Thus we obtain $t_1=1$ and $t=1.$\\

If $b<0$, we use equation \eqref{eq:69} to have
 \begin{equation}\label{eq:85}
 ax^n(x^l)^2-b(abt^2x^n-2)y^2=2.
 \end{equation}
So $(x^l,|y|)$ is a solution of the Diophantine equation
\begin{equation}\label{eq:86}
 ax^nX^2-b(abt^2x^n-2)Y^2=2.
 \end{equation}
Also $(1,t)$ is the minimal positive solution of the Diophantine equation
\begin{equation}\label{eq:87}
 (-abt^2x^n+2)X^2-(-abx^n)Y^2=2.
 \end{equation}
Theorem \ref{thm:4.2} implies
$$b(abt^2x^n-2)=-abx^n,$$
which is impossible.\\

Second, we assume that $x$ is an even integer. \\

 If $b>0$, from equation \eqref{eq:69} we have
 \begin{equation}\label{eq:88}
 \frac{ax^n}{2}(x^l)^2-b\frac{(abt^2x^n-2)}{2}y^2=1.
 \end{equation}
Therefore $(x^l,|y|)$ is a solution of the Diophantine equation
\begin{equation}\label{eq:89}
 \frac{ax^n}{2}X^2-b\frac{(abt^2x^n-2)}{2}Y^2=1.
 \end{equation}
One can easily see that $(t,1)$ is the minimal positive solution of
the Diophantine equation
\begin{equation}\label{eq:90}
 \frac{abx^n}{2}X^2-\frac{(abt^2x^n-2)}{2}Y^2=1.
 \end{equation}
If $ax^n>2$, then by Theorem \ref{thm:4.2} we get
$$ax^n=abx^n,$$
so $b=1.$ Again Theorem \ref{Thm:3.2} gives
$$x^l=t,|y|=1$$
or
$$x^l=3^st,\,\,3\nmid t,\,\,3^s+3=2ax^nt^2.$$
It follows that
$$l=s,\,\,x=3t_1,\,\,t=t_1^s,\,\,n=1.$$
So we have
$$3^{s-1}+1=2at_1^{2s+1},$$
thus $t_1=1$ and $t=1.$  Then we obtain $x=3$, contradicting the fact that $x$ is even. If $ax^n=2$, then we see that $a=n=1,\,\,x=2$. So we get
\begin{equation}\label{eq:91}
2^{2l}-b(bt^2-1)y^2=1.
\end{equation}
But $2bt^2-1+\sqrt{b(bt^2-1)}$ is the fundamental solution of
the Diophantine equation
$$X^2-b(bt^2-1)Y^2=1.$$
So we see that equation \eqref{eq:91} is impossible.\\

 If $b<0$, we use equation \eqref{eq:69} to have
 \begin{equation}\label{eq:92}
 \frac{ax^n}{2}(x^l)^2-\frac{b(abt^2x^n-2)}{2}y^2=1.
 \end{equation}
We deduce that $(x^l,|y|)$ is a solution of the Diophantine equation
\begin{equation}\label{eq:93}
 \frac{ax^n}{2}X^2-\frac{b(abt^2x^n-2)}{2}Y^2=1.
 \end{equation}
Again $(1,t)$ is the minimal positive solution of the Diophantine equation
\begin{equation}\label{eq:94}
 \frac{(-abt^2x^n+2)}{2}X^2-\frac{(-abx^n)}{2}Y^2=1.
 \end{equation}
If $ax^n>2$, then by Theorem \ref{thm:4.2} we get
$$b(abt^2x^n-2)=-abx^n.$$
This is impossible. The proof of the case $ax^n=2$ with $b>0$ is similar.\\

\underline{\bf Case 4:} $c=2.$ First suppose that $x$ is an odd integer. \\

 If $b>0$, from equation \eqref{eq:69} we have
 \begin{equation}\label{eq:95}
 b(abt^2x^n+2)y^2-ax^n(x^l)^2=2.
 \end{equation}
Therefore $(|y|,x^l)$ is a solution of the Diophantine equation
\begin{equation}\label{eq:96}
 b(abt^2x^n+2)X^2-ax^nY^2=2.
 \end{equation}
Also $(1,t)$ is the minimal positive solution of
the Diophantine equation
\begin{equation}\label{eq:97}
 (abt^2x^n+2)X^2-abx^nY^2=2.
 \end{equation}
By Theorem \ref{thm:4.2} we get
$$ax^n=abx^n,$$
so $b=1.$ Again by Theorem \ref{Thm:3.3} we obtain
$$x^l=t,|y|=1$$
or
$$x^l=3^st,\,\,3\nmid t,\,\,3^s-3=2ax^nt^2.$$
So we deduce
$$l=s,\,\,x=3t_1,\,\,t=t_1^s,\,\,n=1.$$
Then we have
$$3^{s-1}-1=2at_1^{2s+1},$$
thus $t_1=1$, and $t=1.$\\

If $b<0$, then equation \eqref{eq:69} implies
 \begin{equation}\label{eq:98}
 b(abt^2x^n+2)y^2-ax^n(x^l)^2=2
 \end{equation}
Thus $(|y|,x^l)$ is a solution of the Diophantine equation
\begin{equation}\label{eq:99}
 b(abt^2x^n+2)X^2-ax^nY^2=2.
 \end{equation}
Again $(t,1)$ is the minimal positive solution of the Diophantine equation
\begin{equation}\label{eq:100}
 (-abx^n)X^2-(-abt^2x^n-2)Y^2=2.
 \end{equation}
If $b(abt^2x^n+2)>1$, then using Theorem \ref{thm:4.2} we get
$$b(abt^2x^n+2)=-abx^n,$$
which is impossible. Hence we have $b(abt^2x^n+2)=1$ and thus
$$b=-1,\,\,a=1,\,\,t=1,\,\,n=1,\,\,x=3.$$
Then we obtain
$$3^{2l+1}+2=y^2,$$
which implies that $(2|3)=1$, contradicting the fact that $(2|3)=-1.$ \\

Now we assume that $x$ is an even integer. \\

 If $b>0$, from equation \eqref{eq:69} we have
 \begin{equation}\label{eq:101}
 \frac{b(abt^2x^n+2)}{2}y^2-\frac{ax^n}{2}(x^l)^2=1.
 \end{equation}
We deduce that $(|y|,x^l)$ is a solution of the Diophantine equation
\begin{equation}\label{eq:102}
 \frac{b(abt^2x^n+2)}{2}X^2-\frac{ax^n}{2}Y^2=1.
 \end{equation}
We know that  $(1,t)$ is the minimal positive solution of
the Diophantine equation
\begin{equation}\label{eq:103}
 \frac{(abt^2x^n+2)}{2}X^2-\frac{abx^n}{2}Y^2=1.
 \end{equation}
By Theorem \ref{thm:4.2}, we get
$$ax^n=abx^n.$$
This implies that $b=1.$ Using Theorem \ref{Thm:3.2}, we obtain
$$x^l=t,|y|=1$$
or
$$x^l=3^st,\,\,3\nmid t,\,\,3^s-3=2ax^nt^2.$$
It follows that
$$l=s,\,\,x=3t_1,\,\,t=t_1^s,\,\,n=1.$$
So we have
$$3^{s-1}-1=2at_1^{2s+1},$$
hence $t_1=1$ and $t=1.$  Then we obtain $x=3$. This contradicts the fact that $x$ is even. \\

If $b<0$, from equation \eqref{eq:69} we have
 \begin{equation}\label{eq:104}
 \frac{b(abt^2x^n+2)}{2}y^2-\frac{ax^n}{2}(x^l)^2=1.
 \end{equation}
Thus $(|y|,x^l)$ is a solution of the Diophantine equation
\begin{equation}\label{eq:105}
 \frac{b(abt^2x^n+2)}{2}X^2-\frac{ax^n}{2}Y^2=1.
 \end{equation}
Also the minimal positive solution of the Diophantine equation
\begin{equation}\label{eq:106}
\frac{(-abx^n)}{2}X^2- \frac{(-abt^2x^n-2)}{2}Y^2=1
 \end{equation}
is $(1,t)$. If $-abx^n>2$ and $b(abt^2x^n+2)>2$, then using Theorem \ref{thm:4.2} we obtain
$$b(abt^2x^n+2)=-abx^n.$$
This is impossible. If $-abx^n=2$, then $a=n=1,b=-1,x=2.$  Thus we have
$$(t^2-1)y^2-2^{2l}=1,$$
which is impossible by considerations modulo $4$.  If $b(abt^2x^n+2)=2$, then one can see that $b=-1,a=t=1,n=2,x=2$ or $b=-1,a=t=n=1,x=4$.
When $b=-1,a=t=1,n=2,x=2$, then from equation \eqref{eq:104} we get
$$y^2-2^{2l+1}=1.$$
Therefore $(|y|,2^l)$ is a solution of the Diophantine equation
$$X^2-2Y^2=1.$$
So we have $y=\pm 3,l=1$ by Theorem \ref{Thm:3.1}. When
$b=-1,a=t=n=1,x=4$, from equation \eqref{eq:104} we get
$$y^2-2\times 4^{2l}=1.$$
It follows that $(|y|,4^l)$ is a solution of the Diophantine equation
$$X^2-2Y^2=1,$$
which is impossible by Theorem \ref{Thm:3.1}.\\

The proof for the cases $c=-4$ and $c=4$ can done exactly by the same way. This completes the proof of Theorem \ref{thm:4.26}.
\end{proof}

Similarly, we have the following

\begin{theorem} {\rm (\cite{yuan3})}\label{thm:4.26} The equation
\begin{equation}\label{eq:107}
\frac{ax^n+c}{abxt^2+c}=by^2, c=\pm 1,\,\pm 2,\,\pm
4,\,2\nmid n,\,\,(a,c)=1
\end{equation}
 holds for some integers $a,b,x,y,n$ and $t$ with $x>1,n>0,t>0,a>0$ if and only if
$$b=t=1\,\,\,or\,\,\,b=|y|=1$$
except for
$$(a,b,x,y,n,t,c)=(1,-1,2,\pm 3,3,1,1).$$
The solutions to equation \eqref{eq:107} are given by Theorems \ref{thm:4.16}, \ref{thm:4.18}, and \ref{thm:4.19} when $b=t=1$. The solutions to equation \eqref{eq:107} are given by $x^{n-1}=t$
when $b=|y|=1.$
 \end{theorem}

 \end{document}